\documentclass[12pt,a4paper]{article}

\usepackage[left=30mm,right=30mm,top=25mm,bottom=30mm]{geometry}
\usepackage{epsfig}
\usepackage{epstopdf}
\usepackage{soul}
\usepackage{xfrac}
\usepackage{comment}
\usepackage[percent]{overpic}

\usepackage{color}
\usepackage{amsthm,amsmath,amssymb}
\usepackage{booktabs}
\usepackage{mathpazo}
\usepackage{microtype}

\usepackage{bm}
\usepackage{sectsty}
\usepackage{subcaption}
\usepackage[
	pdftitle={The inf spectrum of moduli space},
	pdfauthor={Hugo Parlier},
	ocgcolorlinks,
	linkcolor=linkred,
	citecolor=linkred,
	urlcolor=linkblue]
{hyperref}

\definecolor{linkred}{RGB}{112, 11, 75} 
\definecolor{linkblue}{RGB}{216, 56, 241}

\usepackage[hang,flushmargin]{footmisc}
\usepackage{enumitem}
\usepackage{titlesec}
	\titlespacing{\section}{0pt}{12pt}{0pt}
	\titlespacing{\subsection}{0pt}{6pt}{0pt}
	
\titlelabel{\thetitle.\quad}

\makeatletter 

\long\def\@footnotetext#1{%
\H@@footnotetext{%
\ifHy@nesting 
\hyper@@anchor{\@currentHref}{#1}%
\else 
\Hy@raisedlink{\hyper@@anchor{\@currentHref}{\relax}}#1%
\fi 
}}

\def\@footnotemark{%
\leavevmode 
\ifhmode\edef\@x@sf{\the\spacefactor}\nobreak\fi 
\H@refstepcounter{Hfootnote}%
\hyper@makecurrent{Hfootnote}%
\hyper@linkstart{link}{\@currentHref}%
\@makefnmark 
\hyper@linkend 
\ifhmode\spacefactor\@x@sf\fi 
\relax 
}%

\ifFN@multiplefootnote%
\renewcommand*\@footnotemark{%
\leavevmode 
\ifhmode 
\edef\@x@sf{\the\spacefactor}%
\FN@mf@check 
\nobreak 
\fi 
\H@refstepcounter{Hfootnote}%
\hyper@makecurrent{Hfootnote}%
\hyper@linkstart{link}{\@currentHref}%
\@makefnmark 
\hyper@linkend 
\ifFN@pp@towrite 
\FN@pp@writetemp 
\FN@pp@towritefalse 
\fi 
\FN@mf@prepare 
\ifhmode\spacefactor\@x@sf\fi 
\relax%

\makeatother 

\theoremstyle{plain}
\newtheorem{theorem}{Theorem}[section]
\newtheorem{proposition}[theorem]{Proposition}
\newtheorem{lemma}[theorem]{Lemma}

\theoremstyle{definition}
\newtheorem{definition}[theorem]{Definition}

\newtheorem{remark}[theorem]{Remark}

\newcommand{\eps}{{\varepsilon}}
\newcommand{\sys}{{\rm sys}}
\newcommand{\ii}{{\mathit i}}
\newcommand{\Sys}{{\mathfrak S}}
\newcommand{\R}{{\mathbb R}}
\newcommand{\D}{{\mathbb D}}
\newcommand{\Hyp}{{\mathbb H}}
\newcommand{\N}{{\mathbb N}}
\newcommand{\Q}{{\mathbb Q}}
\newcommand{\F}{{\mathbb F}}
\newcommand{\TT}{{\mathbb T}}
\newcommand{\K}{{\mathbb K}}
\newcommand{\Z}{{\mathbb Z}}
\newcommand{\C}{{\mathcal C}}
\newcommand{\GG}{{\mathbb G}}
\newcommand{\GGS}{{\mathbb G}(S)}
\newcommand{\CC}{{\mathcal C}}
\newcommand{\BB}{{\mathcal B}}
\newcommand{\arc}{{\mathcal A}}
\newcommand{\DD}{{\mathcal D}}
\newcommand{\Sp}{{\mathbb S}}
\newcommand{\G}{{\mathcal G}}
\newcommand{\T}{{\mathcal T}}
\newcommand{\M}{{\mathcal M}}
\newcommand{\PSL}{{\rm PSL}}
\newcommand{\area}{{\rm area}}
\newcommand{\arcsinh}{{\,\rm arcsinh}}
\newcommand{\arccosh}{{\,\rm arccosh}}
\newcommand{\diam}{{\rm diam}}
\newcommand{\dias}{{\rm dias}}
\newcommand{\numS}{{\rm Kiss}(S)}
\newcommand{\length}{{\rm length}}
\newcommand{\rank}{{\rm rank}}
\newcommand{\kk}{\Bbbk}
\newcommand{\PP}{\mathcal P}
\newcommand{\FF}{\mathcal F}
\newcommand{\II}{\mathcal I}
\newcommand{\cf}{{\it cf.}}
\newcommand{\ie}{{\it i.e.}}
\newcommand{\curves}{\mathcal H}
\newcommand{\collar}{\mathcal C}
\newcommand{\ZZ}{\mathbb{Z}}
\newcommand{\calP}{\mathcal{P}}
\newcommand{\injrad}{\mathrm{injrad}}
\newcommand{\Homeo}{\mathrm{Homeo}}
\newcommand{\design}{\mathbf{D}_g}
\newcommand{\thick}{\mathbf{T}_g}
\newcommand{\model}{\mathbf{M}_g}
\newcommand{\MCG}{\mathrm{Mod}}
\newcommand{\infspec}{\mathrm{Inf\,Spec}}

\def\bull{\vrule height 1ex width .8ex depth -.2ex}
\renewcommand{\labelitemi}{\bull}

\renewcommand{\footnoterule}{\vfill \rule{0.25\textwidth}{0.5pt} \vspace{4pt}}

\linespread{1.25}

\sectionfont{\large \bfseries}
\subsectionfont{\normalsize}

\setlength{\parindent}{0pt}
\setlength{\parskip}{6pt}

\long\def\symbolfootnote[#1]#2{\begingroup%
\def\thefootnote{\fnsymbol{footnote}}\footnote[#1]{#2}\endgroup}

\def\blfootnote{\xdef\@thefnmark{}\@footnotetext}

\usepackage{mathtools}
\DeclarePairedDelimiter\floor{\lfloor}{\rfloor}
\DeclarePairedDelimiter\ceil{\lceil}{\rceil}
\usepackage[sort,nocompress]{cite}

\begin{document}

{\Large\bfseries The inf spectrum of moduli space}

{\large Ara Basmajian\symbolfootnote[1]{\small 
Supported by a Dolciani faculty research grant, a PSC-CUNY research grant, and a grant from the Simons Foundation (TSM 00013865 A.B.)}, 
Sayantika Mondal, and
Hugo Parlier\symbolfootnote[3]{\small Supported by ANR-SNF Grant number 200021E\_238147
(SUGAR).
}
}

\vspace{0.5cm}
{\bf Abstract.} Given a surface of genus $g \geq 2$, its inf spectrum is the ordered multiset of 
inf lengths of curves on the surface, the infimum being over all hyperbolic metrics. 
We show that the inf spectrum has exponential growth with coefficients that only depend on $g$, improving the earlier double exponential upper bound. The proof involves developing 
a number of tools including the process of ``thickening''   a surface and the notion of 
model surfaces. These tools may be of independent interest.

\vspace{0.5cm}

\section{Introduction} \label{sec:intro}

On a hyperbolic surface of finite type, which we assume to be orientable, any free homotopy class of a closed curve has a unique (geodesic) minimizer. Similarly, any mapping class group orbit of a closed curve also has a minimal representative. Now among all hyperbolic surfaces in a(n) (augmented) Teichm\"uller space, the mapping class group orbit of a closed curve has a uniquely defined metric where the length of this curve is minimal. This metric, using the metaphor that it is stylish to look as short as possible, is called a {\it designer metric}. Note that if the curve is not filling, then the designer metric lies on the boundary of the Deligne-Mumford compactification of moduli space, that is on a noded hyperbolic surface. 

And the collection of these doubly minimal lengths $m_{[\gamma]}$, minimal among representatives of a mapping class group orbit $[\gamma]$ of curves $\gamma$ and among hyperbolic surfaces in a moduli space $\M$, is referred to as the {\it inf spectrum}:
\[
\infspec(\M):= \{m_{[\gamma]} : [\gamma] \mbox{ a mapping class group orbit of }\gamma, \gamma \mbox{ a closed curve} \}.
\]

We study the set of designer metrics and the associated inf spectrum. For simplicity, we focus on the moduli space of closed surfaces of genus $g$, denoted $\M_g$, but our results hold for all finite-type surfaces, namely because all other moduli spaces can be embedded in the compactification of closed ones. 

Our main result is lower and upper bounds on the growth rate of the inf spectrum, addressing a question that arises from \cite{BM}, where the inf spectrum was first defined. 

\begin{theorem}\label{theoremA}
There exists $L_0 >0$ so that for all $L \geq L_0$ the inf spectrum satisfies 
\[
A \frac{e^{L}}{2L} \leq \# \bigl\{[\gamma] \mid m_{[\gamma]} \leq L \bigr\} \leq Be^{C L}
\]
where $A,B$ and $C$ are constants which only depend on $g$. 
The constant $A$ can be taken to be
$$
A= \frac{1}{6} \left( g+ \floor*{\frac{g^2+3}{12}}\right).
$$
\end{theorem}

The exponential growth is of course reminiscent of the growth of closed geodesics on a hyperbolic surface, and this is no coincidence. To show the result, we partition the set of (mapping class group orbits of) curves according to their self-intersection number $n$, and look at their associated designer metrics. The first step in the proof is to ``thicken" the surface with the designer metric to a new metric that has systole lower bounded by $\frac{1}{5n}$ (the 5 in the denominator is an artifact of the proof).
The new metric possibly increases lengths by a multiplicative factor but the factor only depends on topology. This is the essence of the following proposition which we put here as it may be of independent interest. 

\begin{proposition} (Proposition \ref{prop: thickened metric})
For $\gamma$ a closed curve with $i(\gamma,\gamma)=n$, there exists a metric $Y_{\gamma}$ such that $\ell_{\gamma}(Y_{\gamma})\leq c\, m_{\gamma}$ where $\sys(Y_\gamma)\geq \frac{1}{5n}$. The constant $c \geq 1$ depends only on the topology of $X$.
\end{proposition}

Using these thickened surfaces, we next show that the designer metrics can be regrouped to ``behave" like a fixed set of surfaces, which we call model metrics $\model$ (and which will be explained below). Now the counting result follows from a combination of counting the model metrics and counting curves on the model metrics. The main part of the proof is quantifying how the behavior of curves on the model metrics mimics the behavior on the designer metrics.

So what are these model metrics $\model$? For a fixed positive integer $n$, they are associated to topological types of pants decomposition as follows. Let $P=\{\alpha_1,\hdots,\alpha_{3g-3}\}$ be a pants decomposition. To each $\alpha_i$, $i \in \{1,\hdots, 3g-3\}$, we associate a 
$k_i=0,1,2,...,n$ and construct a base surface as follows: we set the length of $\alpha_i$ to be $\frac{k_i}{n}$, and the twists to be $0$. We use the convention that $k_i=0$ corresponds to $\alpha_i$ representing a cusp. Taking the set of all possible model metrics, that is by doing this construction with all possible topological types of pants decomposition, gives us a finite discrete set of surfaces in augmented moduli space, that is including noded hyperbolic metrics. We denote these metrics as $\model^n$ and we denote by $\model$ the set of all model metrics, that is $\model=\cup \model^n$. Note that implicitly, the set $\model^n$ is indexed by the set of topological types of pants decomposition. The sets of topological types of pants decomposition have previously been used to understand the geometry of moduli spaces \cite{Rafi-Tao, Cavendish-Parlier}. As the index $n$ increases, surfaces in $\model^n$ start to probe the thinner parts of moduli space.

We summarize the steps in the proof of Theorem \ref{theoremA} with the following diagram:
\begin{equation*}
 \{\text{Curves}\} \rightarrow \{ \text{Designer metrics} \}\rightarrow \{ \text{Thickened surfaces}\} \rightarrow \{\text{Model surfaces}\}
\end{equation*}
\begin{equation*} 
 \gamma \xmapsto{optimizing \, metric} X_{\gamma} 
 \xmapsto{thickened \, metric} Y_{\gamma} \xmapsto{model \, metric} Z^{P}_\gamma
\end{equation*}

\begin{remark} The following items give a brief condensed description of some of the ancillary results that arise in the investigation of the inf spectrum properties and in the proof of Theorem \ref{theoremA}. More details can be found in 
Sections \ref{sec:inf}, \ref{section:counting figure 8's}, and \ref{sec:propertiesofinfspec}.
\begin{itemize} 
 \item {\bf The constants.}
 The constant $A$ is the number of $\MCG_{g}$ orbits of punctured tori (there is only one) and embedded pairs of pants. See Proposition \ref{prop: infspecproperties}.
The constants $B$ and $C$ can be upper bounded using the final estimates in the proof of Theorem \ref{theoremA}. 

\item {\bf Multiplicity.} The first number  in  the inf spectrum, $0$,    has multiplicity 
$\floor{\frac{g}{2}}+1$, this is the number of $\MCG_g$ orbits of essential simple closed curves. Next is $2\arccosh (3)$ with multiplicity 
\begin{equation}
 2g-2+\floor*{\frac{(g-1)^2}{4}}.
 \end{equation}
This is the number of 
$\MCG_g$ orbits of figure $8$s on the surface. This multiplicity determines the genus of the surface. See Proposition \ref{enum: counting}.

\item {\bf Curve types and inf length.} While there is very little one can surmise about the curve type of a length that appears in the length spectrum, for the inf spectrum there are a number of interesting 
conclusions one can draw knowing the inf length. See Section \ref{sec:propertiesofinfspec}.

\item {\bf Length spectrum comparison.} Comparing the length spectrum of a hyperbolic metric $X$ on the closed surface $\Sigma$ to the inf spectrum of $\Sigma$, we see that the lower bound constant in the inf spectrum is larger since $A >1$. In fact, a more germane comparison is with respect to the spectrum of minimal $X$-length representatives in each mapping class orbit which is a strictly smaller set though has the same growth rate as the length spectrum \cite{BP}. 
\end{itemize}
\end{remark}

{\bf Organization.} In section \ref{sec:prelim} we introduce notation and preliminary results that we use throughout. In section \ref{sec:thick}, we show how to thicken a designer metric (only if needed!) and in section \ref{sec:model}, we introduce and study model metrics. In section \ref{sec:inf}, we prove Theorem \ref{theoremA}. In section \ref{section:counting figure 8's} we count the number of 
$\MCG_g$ orbits of embedded pairs of pants and figure $8$s. Section 
 \ref{sec:propertiesofinfspec} is devoted to studying properties of the inf spectrum, and in section \ref{sec:growth using intersection} we show that using known 
 relations between length and intersection does not give better bounds than our 
 Theorem \ref{theoremA}.

\section{Preliminaries} \label{sec:prelim}

 Unless otherwise stated all curves are primitive and unoriented.
In Table \ref{table:1} we put notation that appears throughout the paper.

\begin{table}[htbp]
	\centering
	\begin{tabular}{|p{0.55\textwidth}|c|c|}
		\hline \hline 
		{\bf Definition} & {\bf Section} &{\bf Notation} \\
		\hline
		\hline 
		Inf length& \ref{sec:intro} & $m_{\gamma}$ \\
		\hline
		Inf spectrum of $\M$ (or $\Sigma$)& \ref{sec:intro} & $\infspec(\M)$ \\
		\hline
		Moduli space& \ref{sec:intro}& $\mathcal{M}:=\mathcal{M}(\Sigma)$ \\
		\hline
		Set of model metrics & \ref{sec:intro} & $\model^n$ \\
		\hline
		$\varepsilon$-thick part of moduli space & \ref{sec:intro} & $\M_g^{\varepsilon}$ \\
		\hline
		Teichm\"uller Space&\ref{subsec:hypsurfaces} & $\T_{g}$ \\
		\hline
		Mapping class group of genus $g$ surface &\ref{subsec:hypsurfaces}& $\MCG_g$\\
		\hline 
		Length function of $\gamma$ on Teichm\"{u}ller space & \ref{subsec:hypsurfaces} & $\ell_{\gamma}$ \\
		\hline
		Systole length of $X$ & \ref{subsec:hypsurfaces} & $\text{sys}(X)$\\
		\hline
		Bers constant & \ref{subsec:collarbers} & $B_g$ \\
		\hline
		Geometric winding number & \ref{subsec:collarbers} & $\text{gw}_{\alpha}(\cdot)$ \\
		\hline
		Collar neighborhood about $\alpha$ with boundary curves of length one & \ref{sec:thick} & $\mathcal{C}_{\alpha}$\\
		\hline
		Half-collar, and truncated pair of pants & \ref{sec:thick} & $\mathcal{H}_x,\ P_x^1$ \\
		\hline
		Intersection number& \ref{sec:thick} & $i(\cdot ,\cdot)$\\
		\hline
		Optimal (Designer) metric for $\gamma$& \ref{sec:thick} & $X_{\gamma}$ \\
		\hline
		Thickened designer metric & \ref{sec:thick} & $Y_{\gamma}$\\
		\hline
		Model metric with associated FN parameter $\vec{k}$ & \ref{sec:model} & $Z^{P}_{\vec{k}}$\\
 \hline
		 number of trivalent graphs & \ref{sec:model} & $t_{g}$\\
 \hline
		\hline
	\end{tabular} 
	\caption{Notations} \label{table:1}
\end{table}

\subsection{Hyperbolic surfaces and their moduli}\label{subsec:hypsurfaces}

We work throughout with \emph{complete, finite-area, orientable hyperbolic surfaces}, i.e.\
Riemannian $2$-manifolds of constant curvature $-1$. Every such surface is conformally
equivalent to a quotient $\Gamma \backslash \Hyp$, where $\Hyp$ denotes the upper half-plane
equipped with the metric $ds^2 = (dx^2+dy^2)/y^2$, and $\Gamma \subset \PSL(2,\R)$ is a
discrete, torsion-free subgroup. We restrict attention to \emph{closed} surfaces of genus
$g \geq 2$, referring to surfaces with cusps or geodesic boundary only when explicitly needed.

\begin{definition}[Teichm\"uller space and moduli space]\label{def:teich}
Fix a smooth closed oriented surface $S$ of genus $g \geq 2$. The \emph{Teichm\"uller space}
$\T_g$ is the space of isotopy classes of hyperbolic metrics on $S$:
\[
	\T_g \;:=\; \bigl\{\text{hyperbolic metrics on }S\bigr\} \big/ {\sim},
\]
where two metrics are equivalent if they are related by an isometry isotopic to the identity.
A point $X \in \T_g$ is called a \emph{marked hyperbolic surface}.

The \emph{mapping class group} $\MCG_g := \MCG(S)$ is the group of isotopy classes of
orientation-preserving homeomorphisms $S\to S$. It acts on $\T_g$ by re-marking: for
$\phi\in\MCG_g$ and $X\in\T_g$, set $\phi\cdot X := X\circ\phi^{-1}$ (with the induced
metric). The quotient
\[
	\M_g \;:=\; \T_g / \MCG_g
\]
is the \emph{moduli space} of genus-$g$ hyperbolic surfaces; its points are isometry classes
of closed hyperbolic surfaces of genus $g$.
\end{definition}

\begin{definition}[Fenchel--Nielsen coordinates]\label{def:FN}
A \emph{pants decomposition} of $S$ is a maximal collection
$P = \{\gamma_1,\ldots,\gamma_{3g-3}\}$ of isotopy classes of pairwise disjoint, pairwise
non-isotopic simple closed curves. The complement
$S\setminus(\gamma_1\cup\cdots\cup\gamma_{3g-3})$ consists of $2(g-1)$ pairs of pants
(three-holed spheres).

Given a pants decomposition $P$, every marked surface $X\in\T_g$ is completely determined by
its \emph{Fenchel--Nielsen coordinates}
\[
	\bigl(\ell_{\gamma_1}(X),\ldots,\ell_{\gamma_{3g-3}}(X);\;
	\tau_{\gamma_1}(X),\ldots,\tau_{\gamma_{3g-3}}(X)\bigr)
	\;\in\; \R_{>0}^{3g-3}\times\R^{3g-3},
\]
where $\ell_{\gamma_i}(X)>0$ is the hyperbolic length of the geodesic representative of
$\gamma_i$ in $X$, and $\tau_{\gamma_i}(X)\in\R$ is the corresponding twist parameter.
The Fenchel--Nielsen map is a real-analytic diffeomorphism
$\T_g \xrightarrow{\;\sim\;}\R_{>0}^{3g-3}\times\R^{3g-3}$.
\end{definition}

\begin{definition}[Curve lengths and systole]\label{def:length}
Every free homotopy class of essential closed curves on $X$ contains a unique closed geodesic
representative. For a closed curve $\gamma$, we write $\ell_\gamma(X)$ for the hyperbolic
length of its geodesic representative on $X$. The function
$\ell_\gamma\colon\T_g\to\R_{>0}$ is real-analytic.

The \emph{systole} of $X$ is
\[
	\sys(X) \;:=\; \min_\gamma \ell_\gamma(X),
\]
the minimum taken over all (free homotopy classes of) essential closed curves. For
$\varepsilon>0$, the \emph{$\varepsilon$-thick part} of moduli space is
\[
	\M_g^\varepsilon \;:=\; \bigl\{X\in\M_g : \sys(X)\geq\varepsilon\bigr\}.
\]
\end{definition}

For $\gamma$ an essential closed curve on 
$\Sigma$, the $\text{inf}$ length of $\gamma$ is defined to be
$$
m_{\gamma}:=
\text{inf} \{\ell_{\gamma}(X): X \in \T_g \}.
$$
Curves in the same mapping class group orbit have equal inf lengths. This is the reason why we consider $m_{[\gamma]}$, where $[\gamma]$ is in the $\MCG_g$ orbit of $\gamma$.
Moreover, this infimum is achieved at a unique marked hyperbolic structure denoted $X_{\gamma}$ in moduli space. For basics on inf lengths and the inf spectrum 
see \cite{BM}.

\subsection{The collar lemma and Bers' theorem}\label{subsec:collarbers}

The collar lemma states that a simple closed geodesic of length $\ell$ on a hyperbolic surface has a collar of width (that is, the distance from any point of the collar to the geodesic is) 
$$
\arcsinh \left(\frac{1}{\sinh{\frac{\ell}{2}}}\right)
$$

In particular, simple geodesics of length at most one cannot intersect. 

\begin{theorem}\label{thm:Bers}
For each $g \geq 2$, there exists a constant $B_{g}$ so that any hyperbolic surface $X$ of genus $g$ has a pants decomposition whose pants curves have length less than $B_{g}$. Moreover, the decomposition can be chosen to include all curves of length at most $1$ on $X$. \end{theorem}

The constant $B_{g}$ is called the Bers constant and we call a pants decomposition that satisfies the above inequality a short pants decomposition. See \cite{BuserBook}.

The following is known as the Wolpert lemma in the literature \cite{BuserBook}. 

\begin{lemma}\label{lem:WolpertLemma}
Suppose $\phi : X \rightarrow Y$ is a K-quasiconformal homeomorphism between two hyperbolic surfaces. Then
$$
\frac{\ell_{\gamma}(X)}{K} \leq \ell_{\phi(\gamma)}(Y) \leq K \ell_{\gamma}(X)
$$
for any closed curve $\gamma$ on $X$. 
\end{lemma}

The following proposition has a much more general formulation but we state it here for a pair of pants since that is the setting for which we need it. This formulation can be proven using elementary hyperbolic geometric formulas involving pairs of pants, see \cite{BuserBook}. Fix $a,b \geq 0$, consider the pair of pants $P_x$ with cuff lengths $a$, $b$, and $x \geq 0$, and denote 
$P_x$ truncated along the equidistant curve of length 1 by $P_x^1$.

\begin{proposition} \label{prop: geometricconvergence}
The pants $P_x$ geometrically converges to 
 $P_0$, as $x$ goes to $0$. In particular, there is a $C_{x}$-biLipschitz mapping 
 from the truncated $P_x^1$ to the truncated $P_{0}^1$ which 
 \begin{enumerate}
 \item when restricted to the boundary curves is an isometry 
 \item respects the endpoints of a simple orthogeodesic going from one boundary curve to another. 
 \end{enumerate}
 The constant $C_{x}$ is monotone and goes to 1, as $x$ goes to $0$.
\end{proposition}

The next Lemma and Proposition follow from the work of Bishop, see \cite{Bishop2002}.

\begin{lemma}\label{lem:bishop}
If two pairs of pants have an upper bound on their pants curves, and an upper bound on the ratios of lengths between their boundary curves, then there is a quasiconformal mapping between them with a constant that only depends on these two bounds. Moreover, the quasiconformal homeomorphism can be chosen so that it is affine on the boundary geodesics. 
\end{lemma}

\begin{proposition}\label{prop:bishop}
Suppose $X$ and $Y$ are hyperbolic surfaces and there is a pants decomposition for which both $X$ and $Y$ have upper bounded pants curves and upper bounded ratios of pants lengths, then 
there is a quasiconformal homeomorphism between $X$ and $Y$ with a constant that just depends on these two bounds. 
\end{proposition}

\begin{remark}\label{rem:embedding}
While we focus on the moduli space $\M_g$ of closed surfaces of genus $g$, the moduli space of complete finite area of surfaces of finite-type can be embedded in the closure $\overline{\M_g}$ for a well chosen $g$. Specifically, the moduli space of a genus $h$ surface with $m$ cusps can be embedded in $\overline{\M_g}$ where $g=h +\lceil\frac{m}{2}\rceil$. It follows that our upper bounds on the growth of the inf spectrum for $\M_{h,m}$ are direct consequences of our upper bounds for $\M_g$. 
\end{remark}


\section{Thickening a surface}\label{sec:thick}
We suspect that any designer metric for a curve with at $n$ self-intersections has systole at least on the order $\frac{1}{n}$, but the only inequality that is somewhat immediate is a lower bound on the order of $\frac{1}{n^2}$. Nevertheless, the method we use here is somewhat general, and allows us to modify any surface with systole less than $\frac{1}{5n}$ to obtain a surface of systole at least $\frac{1}{5n}$ while controlling the length of certain curves. 

When the designer hyperbolic surface associated to a curve $\gamma$ has a small systole, we can find a different surface with a bigger systole (thickened surface) such that the length of the curve in this surface is comparable to its inf length up to a multiplicative constant.

Let $X_\gamma$ be the inf metric for $\gamma$ and let $\sys(X_\gamma)$ denote the systole of $X_\gamma$. If $\sys(X_\gamma) \geq \frac{1}{5n}$, we leave the surface as it is; but otherwise, we modify the surface to obtain a surface $Y_\gamma$ which has a systole of length at least $\frac{1}{5n}$ and with very similar properties to $X_\gamma$. We describe the process in what follows.

Consider all curves $\alpha$ of length $\leq \frac{1}{5n}$ on $X_\gamma$. By the collar lemma, they are disjoint, and we consider a short pants decomposition to which they belong, as in Theorem \ref{thm:Bers}. We now define the surface $Y_\gamma$ as the surface obtained by replacing all curves $\alpha$ with curves of length $\frac{1}{n}$. Specifically, this means not changing any of the other Fenchel-Nielsen parameters while making the length of each such $\alpha$ equal to $\frac{1}{n}$ and choosing the twist parameter to be $0$ (but this is not important). 

We now proceed to show that the resulting surface has nice properties, namely that the length of $\gamma$ on $Y_\gamma$ is (at worst) the length of $\gamma$ on $X_\gamma$ times a constant that only depends on the topology of the surface. This constant will be shown to depend only on topology.

Note that if there are multiple curves $\alpha$ that require thickening, we perform the procedure one curve after another, so implicitly we decide on an arbitrary order of the set of these short curves and the following procedure is done for one curve $\alpha$ at a time. 

We consider a short curve $\alpha$, a short pants decomposition and a collar $\collar_{\alpha}$ bounded by two equidistant curves of length exactly $1$. Let $P$ and $Q$ be the two pairs of pants on which $\alpha$ lives. (If $\alpha$ lives inside a one-holed torus and is thus the boundary of one pair of pants and not two, we proceed in the exact same way.) 

We now break up $\gamma$ on the union $P \cup Q$ into arcs which lie in the different subsurfaces, and are one of the following types (see Figure \ref{fig:arc types}):
\begin{enumerate}[label=\Roman*.]
 
 \item Arcs in $\collar_{\alpha}$ which go from one boundary curve back to the same boundary curve. These do not intersect $\alpha$ and are not necessarily simple. 
 \item Arcs that lie on $P\setminus \collar_{\alpha}$ or $Q\setminus \collar_{\alpha}$. 
 \item Arcs in $\collar_{\alpha}$ which go from one boundary curve to the other. These are all simple arcs, and all intersect $\alpha$. They may have pairwise intersection. 
\end{enumerate}

\begin{figure}[h]
 \centering
 \vspace{-10pt}
 \begin{overpic}[width=0.75\linewidth]{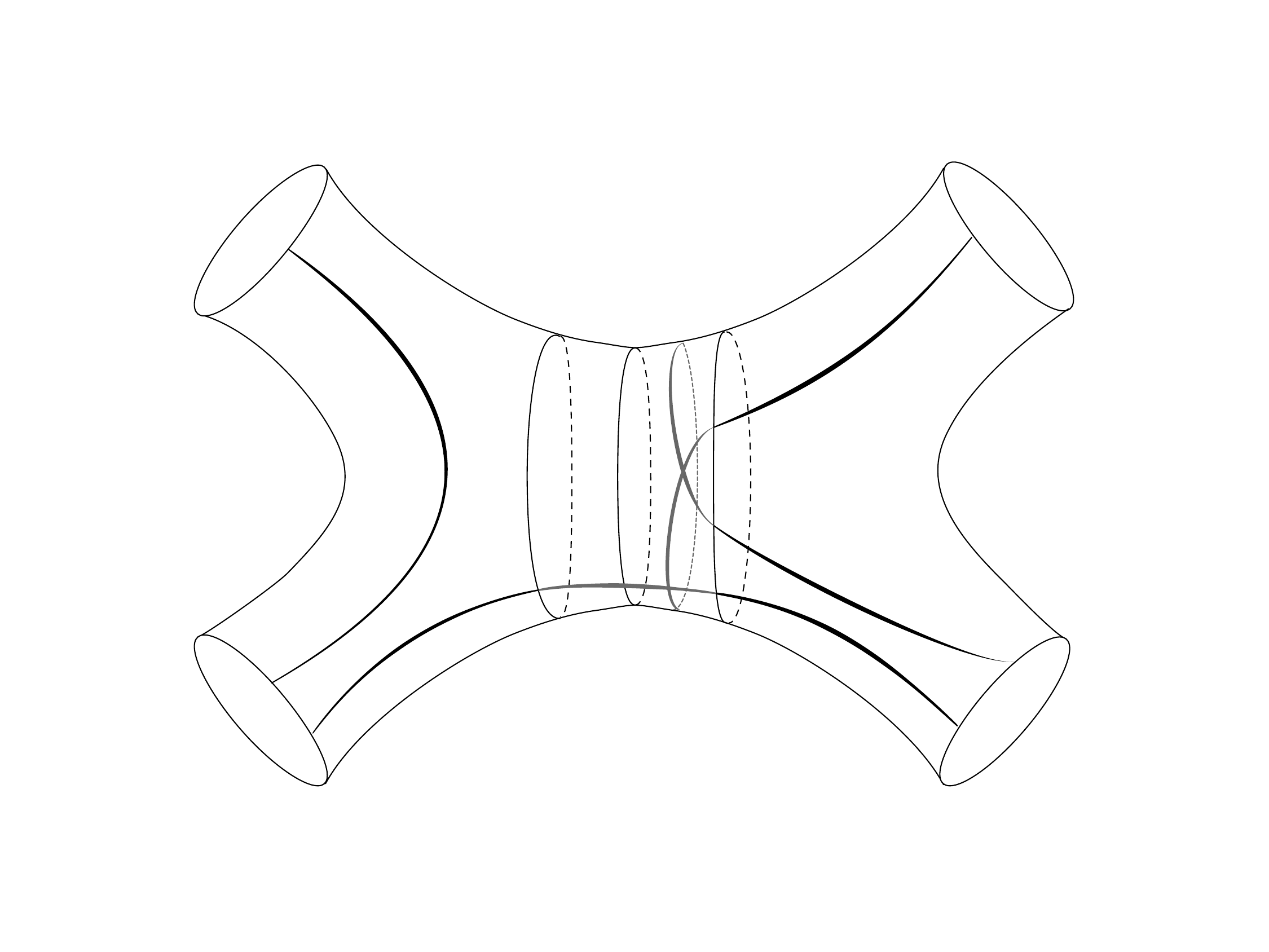}
 \put(53,41){I}
 \put(28,50){II}
 \put(75,50){II}
 \put(46,30){III}
 \end{overpic}
 \vspace{-10pt}
 \caption{Arcs of Type I, II and III}
 \label{fig:arc types}
\end{figure}

We rebuild $\gamma$ on $Y_\gamma$ by finding the corresponding arcs. The only subtlety here is on the choice of endpoints of the arcs. Along the boundary curves of $P$ or $Q$ the choice is canonical (we haven't changed the Fenchel-Nielsen parameter). For endpoints on the boundary of $\collar_{\alpha}$, we make the choice you might have guessed: we take orthogeodesics between these boundary curves and the boundary curves of $P$ or $Q$ to define starting points and then parameterize by length. 

The lengths of the arcs of the three types behave well. Arcs of type I and type II are controlled via a geometric convergence phenomena (Lemmas \ref{lem:geomcollars} and \ref{lem:geompants}), whereas arcs of type III are in fact strictly decreased. 
Arcs of type I and II are in a specific pair of pants, hence for Lemmas \ref{lem:geomcollars} and \ref{lem:geompants} we work in a pair of pants.

 {\bf Notation:} 
 Fix $a,b >0$ and let $P_x$ be a pair of pants with cuffs of length $a$, $b$ and a cuff $\alpha$ of length $x$ where $0 \leq x \leq 1$. 
 Denote by $\mathcal{H}_x \subset P_{x}$ the half-collar bounded by an equidistant curve of length $1$ and the geodesic of length $x$, $0 \leq x \leq 1$. 
 Next for $x<t \leq 1$, denote the equidistant neighborhood of boundary length $t$ deleted from $\mathcal{H}_x$ by $\mathcal{H}_{x}^t \subset \mathcal{H}_x$. Finally, denote by $P_{x}^{1}$ the pair of pants $P_x$ with $\mathcal{H}_x$ deleted; that is,
 $P_{x}^{1}$ is the pair of pants with the collar neighborhood about $\alpha$ of boundary length 1 deleted.

\begin{lemma}(Type I) \label{lem:geomcollars}
The hyperbolic cylinders $\mathcal{H}_{x}^{t}$ and $\mathcal{H}_0^t$ are $C_x$-biLipschitz equivalent by a mapping which when restricted to either boundary curve is an isometry. 
If $0 \leq x,y <t \leq 1$, then the total length of the arcs of type I in $\mathcal{H}_{x}^{t}$ is not increased by more than a multiplicative factor of $C_1^2$
by a biLipschitz mapping $\phi : \mathcal{H}_{x}^{t} \rightarrow \mathcal{H}_{y}^{t}$ which when restricted to the boundary 
of $\mathcal{H}_{x}^{t}$ is an isometry.

\end{lemma}

\begin{figure}[h]
 \centering
 \vspace{-20pt}
 \begin{overpic}[trim={1cm 1cm 10cm 5cm}, clip, width=0.6\linewidth]{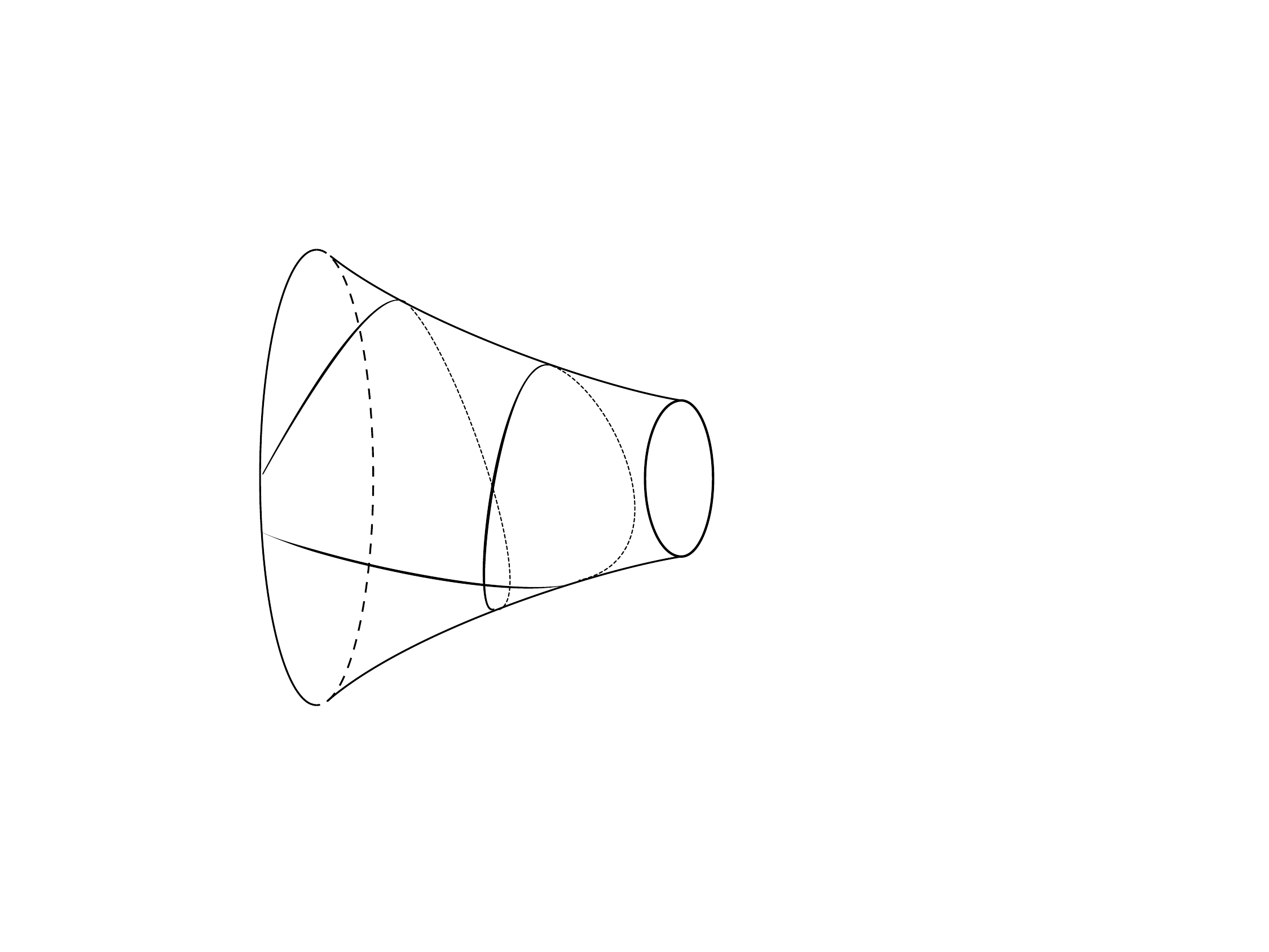}
 \put(73,35){$x$}
 \put(27,20){$1$}
 \put(43,50){$H_x$}
 \end{overpic}
 \vspace{-50pt}
 \caption{Type I: Arcs in $C_\alpha$ from one boundary back to the same, not intersecting $\alpha$}
 \label{fig:type I}
\end{figure}

\begin{proof}
 As $x$ goes to $0$, there is geometric convergence of $\mathcal{H}_x$ to $\mathcal{H}_0$. Restricting to the truncated compact regions 
 $\mathcal{H}_{x}^{t}$ and $\mathcal{H}_0^t$ implies the result (see Proposition \ref{prop: geometricconvergence}). 
 Since the total self-intersection of $\gamma$ is $n$, the arcs of type I remain in the truncated part of the cusp collar
 (see for instance \cite{basmajian2025primeorthogeodesics}), say in $\mathcal{H}_0^t$, for some $t>0$ that is a function of $n$. Using the bi-Lipschitz equivalence twice,
 first from $\mathcal{H}_x^t$ to $\mathcal{H}_0^t$, and then from $\mathcal{H}_0^t$ to $\mathcal{H}_y^t$ we have that 
 $\mathcal{H}_x^t$ is $C_{x}C_{y}$-biLipschitz equivalent to $\mathcal{H}_y^t$. Finally, $C_{x}C_{y} \leq C_1^2$ yields the result. 

\end{proof}

\begin{lemma}[Type II] \label{lem:geompants}
 For all $x$ where $0 \leq x \leq 1$, $P_x^1$ and $P_0^1$ are $D_x$-biLipschitz equivalent by a biLipschitz map that is an isometry on the boundary curves. If $0 \leq x,y \leq 1$, then the total length of the arcs of type II in $P_x^1$ is not increased by more than a multiplicative factor of $D_1^2$ by a biLipschitz mapping $\psi: \mathcal{P}_{x}^{1} \rightarrow \mathcal{P}_{y}^{1}$ which when restricted to the boundary of $\mathcal{P}_{x}^{1}$ is an isometry.
\end{lemma}

\begin{figure}[h]
 \centering
 \vspace{-20pt}
 \begin{overpic}[width=0.6\linewidth, trim=5cm 4cm 5cm 4cm, clip]{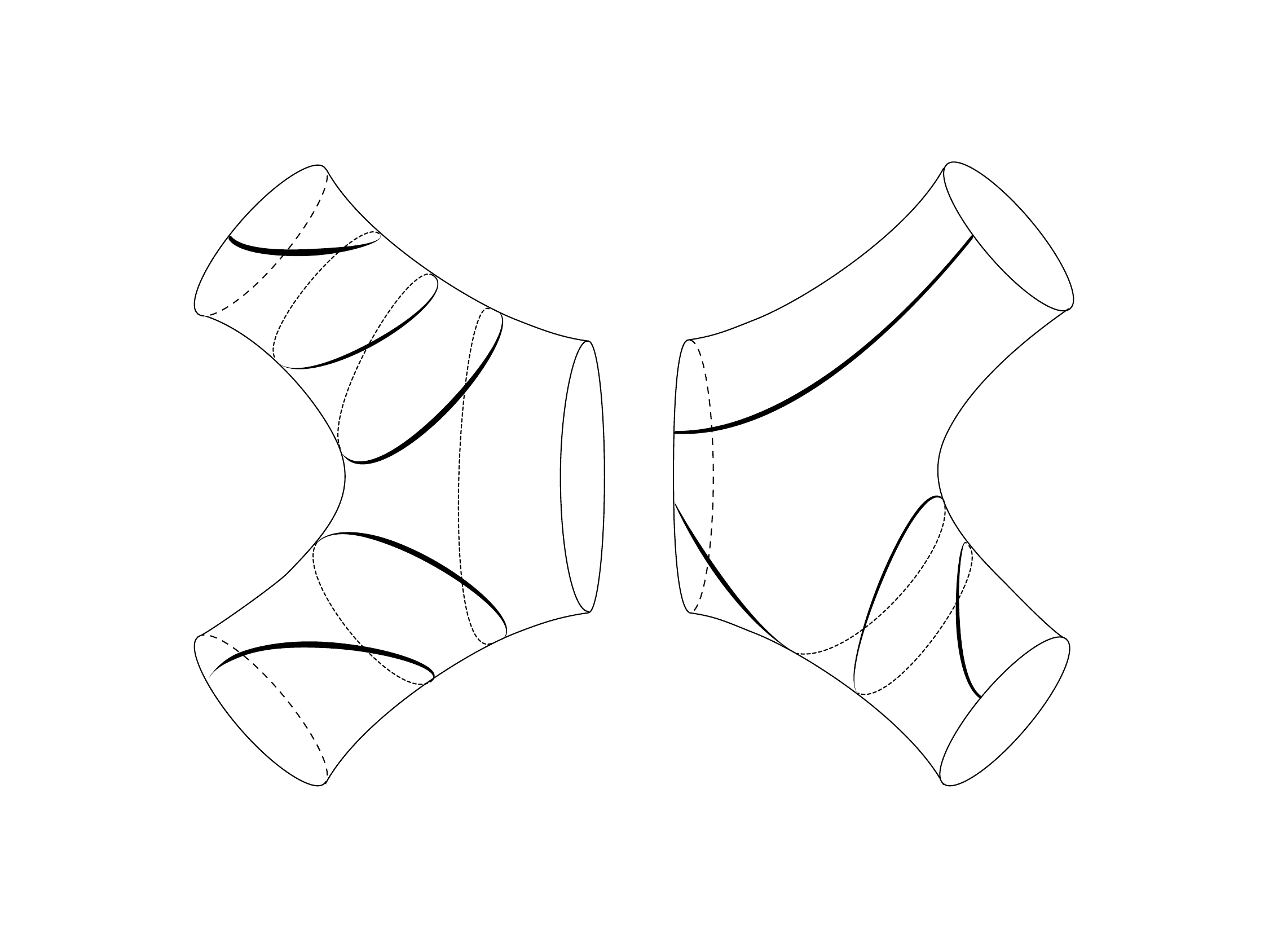}
 \put(47,17){$1$}
 \put(33,41){$P_x^1$}
 \end{overpic}
 \vspace{-10pt}
 \caption{Type II: Arcs in $P^{1}_x$}
 \label{fig:type II}
\end{figure}

\begin{proof}
As $x$ goes to $0$, $P_x$ geometrically converges to $P_0$. Now restricting to the compact regions $P_x^1$ and $P_0^1$ implies the desired biLipschitz equivalence, see Proposition \ref{prop: geometricconvergence}. By composing the biLipschitz map from $P_{x}^{1}$ to $P_{0}^1$ with the biLipschitz map from $P_{0}^1$ to $P^{1}_{y}$, we obtain a biLipschitz map from $P_{x}^{1}$ to $P^{1}_{y}$
with constant $D_{x} D_{y}$ and the result follows.
\end{proof}

Denote by $\mathcal{C}_{\alpha}(x)$ the collar about the geodesic $\alpha$ of length $x \leq 1$ with boundary equidistant curves of length 1. Note that $\mathcal{C}_{\alpha}(x)$ is uniquely determined by the length of $\alpha$. Let 
$\mathcal{A}$ be the set of $\gamma$ arcs 
 passing through $\mathcal{C}_{\alpha}(x)$ from one boundary component to the other. We denote the length of 
 an arc $s$ in $\mathcal{C}_{\alpha}(x)$ by $\ell_{s}(x)$.

\begin{definition}
Let $X$ be a closed hyperbolic surface and $\alpha$ a simple closed geodesic of length $x$ and $s$ a geodesic arc that enters and leaves the collar $\mathcal{C}_{\alpha}(x)$. Then the {\it geometric winding of $s$ about $\alpha$}

$$
gw_{\alpha}(s):=\biggl\lfloor \frac{L}{\ell_{\alpha}(X)} \biggr\rfloor
$$
where $L$ is the length of the orthogonal projection of $s$ onto $\alpha$.

\end{definition}

\begin{lemma}(Type III) \label{lem: typeIII}
 Suppose $\gamma$ is a closed curve with self-intersection number $n$ with geodesic arcs $\mathcal{A}$ passing through 
 $\mathcal{C}_{\alpha}(x)$, where $x \leq \frac{1}{5n}$. 
 Let $\mathcal{A}^{\prime}$ be any set of geodesic arcs passing through $\mathcal{C}_{\alpha}(\frac{1}{n})$ with the same number of arcs and total geometric winding number as $\mathcal{A}$. Then 

\begin{equation} \label{eq: sumofstrands}
\sum_{s^{\prime} \in \mathcal{A}^{\prime}} \ell_{s^{\prime}} \left(\frac{1}{n}\right)
\leq \sum_{s \in \mathcal{A}} \ell_{s}\left(x \right).
\end{equation}
\end{lemma}

\begin{figure}[h]
 \centering
 \begin{overpic}[trim={1cm 1cm 1cm 5cm}, clip, width=0.6\linewidth]{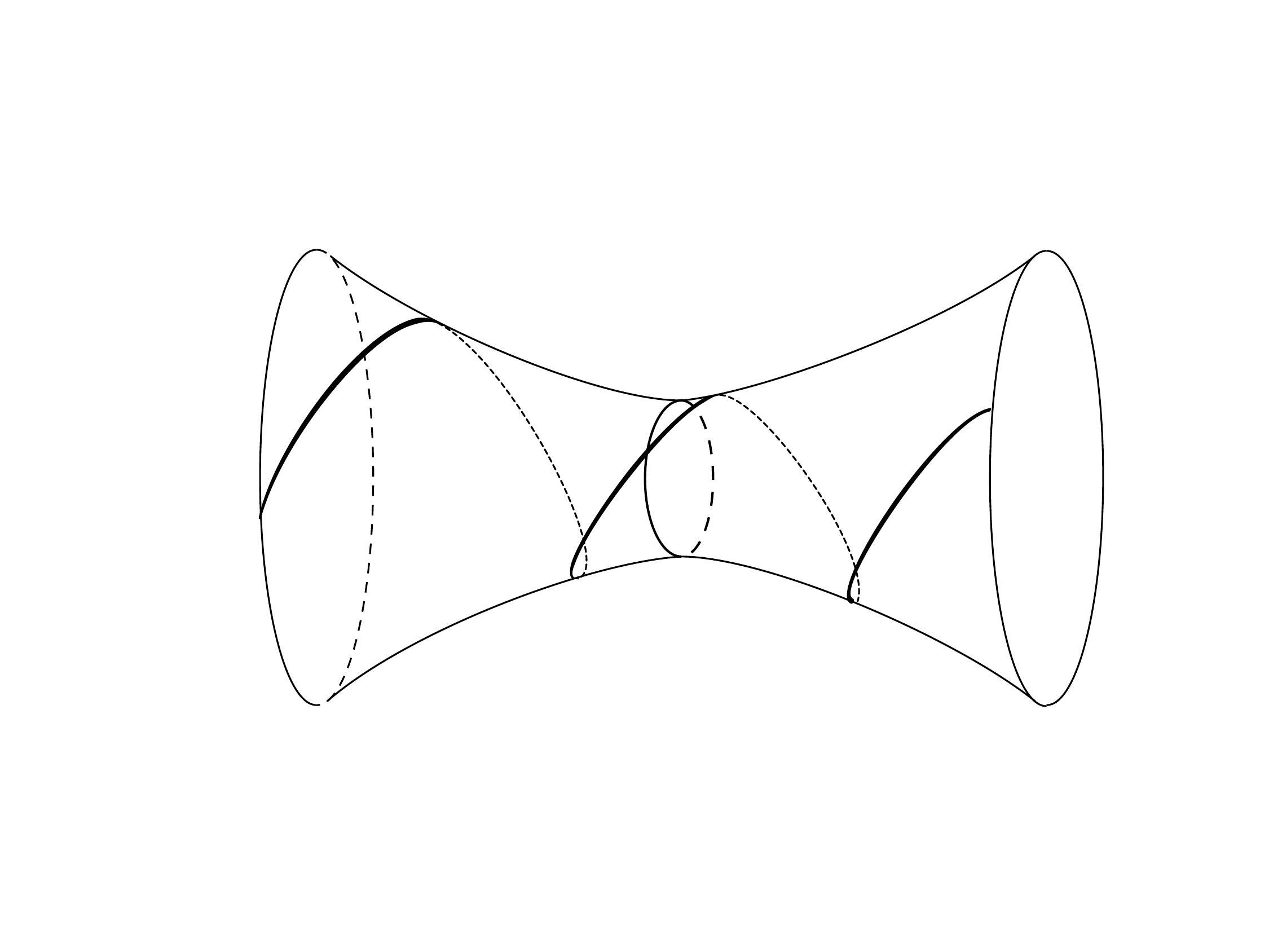}
 \put(20,15){$1$}
 \put(86,15){$1$}
 \put(51,25){$x$}
 \put(50,18){$C_{\alpha}(x)$}
 \end{overpic}
 \vspace{-20pt}
 \caption{Type III: Arcs in $C_\alpha$ from one boundary to the other, intersecting $\alpha$}
 \label{fig:type III}
\end{figure}

\begin{proof} 
Let $m=|\mathcal{A}|=|\mathcal{A}^{\prime}|$ be the number of strands. 
If $m=0$ then there is nothing to prove. If $m=1$, then Wolpert's 
infinitesimal twist formula (see \cite{BuserBook}) dictates that the single strand that passes through $\alpha$ must intersect $\alpha$ at a right angle with respect to the 
optimal metric $X_{\gamma}$. 

Next, suppose there are $m \geq 2$ strands. First note that without loss of generality we can assume (by Dehn twisting) that one of the strands has geometric winding number about $\alpha$ being 0. That is, the strand makes less than a full turn about 
$\alpha$. Now, since the self-intersection number of $\gamma$ is $n$ the sum

\begin{equation}\label{eq: sum}
\sum_{s^{\prime} \in \mathcal{A}^{\prime}} \text{gw}_{\alpha}(s^{\prime}) \leq n.
\end{equation}

Recall that the collar width of $\mathcal{C}_{\alpha}(\frac{1}{n})$ is 
$2 \arccosh \left( n \right) \leq 2 \log n +2\log 2$. Now note that an arc $s^{\prime} \in \mathcal{A}^{\prime}$ which traverses the collar can be upper bounded by twice the width of the collar plus 
$g\text{}w_{\alpha}(s^{\prime}) (\frac{1}{n})$ plus $\frac{1}{2}$ to allow for traversing boundary. That is, a piecewise geodesic comprised of an orthogeodesic segment, followed by traversing the geodesic 
$\alpha$ a $ \text{gw}_{\alpha}(s^{\prime})$ number of times, followed by another orthogeodesic segment, and finally an arc of length less than $\frac{1}{2}$ along the length one boundary. This yields

\begin{equation}
\ell_{s^{\prime}}\left(\frac{1}{n}\right) 
\leq 2 \log n +2 \log 2 + \text{gw}_{\alpha}(s^{\prime})\cdot \frac{1}{n} +\frac{1}{2}.
\end{equation}

Using inequality (\ref{eq: sum}) and summing over the geodesic arcs we have

\begin{align}
\sum_{s^{\prime} \in \mathcal{A}^{\prime}}\ell_{s^{\prime}}\left(\frac{1}{n}\right)
&\leq 2m \log n + 2m \log 2 + 1 +\frac{m}{2} = 2m \log n + m \left(2\log 2 +\frac{1}{2}\right)+ 1 \\
\label{eq:upperbound} &\leq 2m \log n + 2m+ 1 \leq 2m\log n + 2m \log 5 =2m \log 5n \\
& \leq 2m \arccosh \left(5n \right) = \sum_{s \in \mathcal{A}} 2 \arccosh \left(5n \right) \leq \sum_{s \in S} \ell_{s}\left(\frac{1}{5n} \right) \\
&\leq \sum_{s \in \mathcal{A}} \ell_{s}\left(x \right)
\end{align}

In the second inequality of expression (\ref{eq:upperbound}) we have used the fact that $m \geq 2$.
The next to last inequality follows from the fact that a strand in $X$ passing through the collar must have length at least 
 the width of the collar which is $2 \arccosh \left(5n \right)$. Finally, the last inequality follows from our assumption that 
 $x \leq \frac{1}{5n}$.
\end{proof}

\begin{lemma} (local surgery)\label{lem:strict} Suppose $X$ is a hyperbolic surface with a short pants decomposition 
 having a pants curve $\alpha$ of length less than $\frac{1}{5n}$. Then the metric $Y$ with all Fenchel-Nielsen coordinates the same as that for $X$ except for the $\alpha$ curve for which $\ell_{\alpha}(Y)=\frac{1}{n}$
satisfies 
\begin{equation}
 \ell_{\gamma}(Y)\leq C\ \ell_{\gamma}(X) \text{, for any curve $\gamma$ with $i(\gamma, \gamma)=n$}.
\end{equation}
The constant $C \geq 1$ depends only on the topology of $X$.
\end{lemma}

\begin{proof} Let $\gamma \subset X$ have intersection number $n$, and suppose the curve $\alpha$ bounds the pairs of pants $P$ and 
$Q$. Let $X^{\prime}$, resp. $Y^{\prime}$, be the geometry of $X$, resp. $Y$, outside $P \cup Q$. By assumption
$X^{\prime}=Y^{\prime}$, and hence $\ell_{\gamma}(X^{\prime})=\ell_{\gamma}(Y^{\prime})$. As a result, we need only focus on the $\gamma$ arcs in $P \cup Q$. Moreover, since arcs of type I or II are contained in one of the pants $P$ or $Q$, we will argue on the pair of pants $P$, the same holds for the type I or II arcs on $Q$.

Arcs of type I are controlled by Lemma \ref{lem:geomcollars}. There is a $C$-biLipschitz map $\phi : H_{x}^{t} \rightarrow H_{\frac{1}{n}}^t$ which is an isometry on the length 1 boundary curve. Here $t$ is a function of $n$. If $s$ is such an arc
then the biLipschitz image $\phi (s)$ has the same relative position on the boundary of length 1 as $s$ does.

Arcs of type II are controlled by Lemma \ref{lem:geompants}. There is a $C$-biLipschitz map $\psi : P_{x}^{1} \rightarrow P_{\frac{1}{n}}^1$ which is an isometry on the boundary of $P_{x}^{1}$. If $s$ is such an arc, then $\psi (s)$ has the same relative position on the boundary of $P_{x}^{1}$ as $s$ does. Here, the biLipschitz constant depends on the lengths of the other two boundary geodesics of the pair of pants $P$. Thus, the biLipschitz constant may be different for $Q$, but we will take the larger of the two. In the end, it will not matter since all the pants curves have length bounded above by the Bers constant. 

The union of the type III arcs is the set $\gamma \cap \mathcal{C}_{\alpha}(x)$. For each type III arc $s$ in $\gamma \cap \mathcal{C}_{\alpha}(x)$, we construct a geodesic arc 
$s^{\prime} \in \mathcal{C}_{\alpha}(\frac{1}{n})$ that starts and ends at the same boundary points as $s$ did, and which has the same geometric winding number around $\alpha$ as $s$ does. By Lemma \ref{lem: typeIII} we have
\begin{equation}
\sum_{s^{\prime} \in \gamma \cap \mathcal{C}_{\alpha}(\frac{1}{n})} \ell_{s^{\prime}} \left(\frac{1}{n}\right)
\leq \sum_{s \in \gamma \cap \mathcal{C}_{\alpha}(x)} \ell_{s}\left(x \right).
\end{equation}

We now build a curve $\gamma^{\prime} \subset Y$ in the same homotopy class as $\gamma$. Outside of the two pairs of pants 
that bound $\alpha$, we let $\gamma^{\prime}=\gamma$. Hence, $\ell_{\gamma^{\prime}}(Y^{\prime})=\ell_{\gamma}(X^{\prime})$.

Choose an orientation on the geodesic $\gamma \subset X$ and an initial point. Then $\gamma$ decomposes into a
 union of ordered geodesic arcs $s_{1}, s_{2},..., s_{m}$ in $X$; that is,
 $\gamma = s_{1}*s_{2}*...*s_{m}$
 where an arc $s_{i}$ is either a connected segment of $\gamma$ in $X-\{PUQ \}$, a type I arc in $P$ (or in $Q$),
 a type II arc in $P$ (or in $Q$), or a type III arc in $P \cup Q$.

 We next define a curve $\gamma^{\prime} \subset Y$ using the biLipschitz mappings $\phi$ and $\psi$ above. 

 Define

\begin{equation}
\Phi (s_{i})= 
\begin{cases}
 s_{i} , \text{ if } s_{i} \subset X^{\prime}\\
 \phi (s_i), \text{ if } s_{i} \text{ is type I in } P \cup Q \\
 \psi (s_i), \text{ if } s_{i} \text{ is type II in } P \cup Q \\
 s_{i}^{\prime}, \text{ if } s_{i} \text{ is type III in } P \cup Q 
\end{cases}
\end{equation}

Then our curve in $Y$ is 

\begin{equation}
 \Phi (\gamma)=\Phi(s_{1})*\Phi(s_{2})*...*\Phi(s_{m}) 
\end{equation}

Since the biLipshitz mapping $\phi$ respects the relative position of a point on the boundary of the truncated half-collar 
and $\psi$ respects the relative position of a point on the boundary of the truncated pair of pants, our constructed curve 
$\Phi (\gamma)$ is continuous and in the same homotopy class as $\gamma$. Moreover, by construction, at worst $\Phi$ increases the length of an arc by the maximum of the two biLipshitz constants. Thus we have

\begin{equation}
\ell_{\gamma}(Y)\leq 
\sum_{i=1}^{m} \ell_{\Phi(s_{i})}(Y)
\leq C \sum_{i=1}^{m} \ell_{s_{i}}(X)
=C \ell_{\gamma}(X)
\end{equation}
The first inequality is from the fact that the $Y$-length of the geodesic in the homotopy class of $\gamma$ is at most 
the $Y$-length of $\Phi (\gamma)$. Finally, the constant $C \geq 1$ depends only on the topology of $X$.
\end{proof}

\begin{proposition} \label{prop: thickened metric}
For $\gamma$ a closed curve with $i(\gamma,\gamma)=n$, there exists a metric $Y_{\gamma}$ such that $\ell_{\gamma}(Y_{\gamma})\leq c\, m_{\gamma}$ where $\sys(Y_\gamma)\geq \frac{1}{5n}$. The constant $c \geq 1$ depends only on the topology of $X$.
\end{proposition}

\begin{proof} Let $\gamma$ be a closed curve with $i(\gamma,\gamma)=n$. 
By the collar lemma, there are at most $3g-3$ curves of length less than $\frac{1}{5n}$ on $X_{\gamma}$ which we denote as 
$\alpha_1,\alpha_2,..., \alpha_m$; $0 \leq m \leq 3g-3$. Take a short pants decomposition which necessarily includes these curves. Using Lemma \ref{lem:strict}, we perform surgery on $\alpha_1$ to obtain the metric $Y_1$ which has 
$\ell_{\alpha_1} (Y_1)=\frac{1}{n}$ and the geometry in the complement of the pants containing $\alpha_1$ being unchanged. In particular, the lengths of the other pants curves remain the same. We continue by performing this surgery on $\alpha_2$, to obtain $Y_2$ and continue $m$-times to achieve the final metric $Y_{\gamma}$ with $\sys (Y_{\gamma}) \geq \frac{1}{5n}.$ 
If the surgered curves do not bound the same pair of pants in the short decomposition then their respective bilipschitz maps are independent of each other, and the multiplicative constant $C$ in Lemma \ref{lem:strict} remains the same. Since there are at most 3 surgered curves that can bound a pair of pants, the multiplicative constant is at most 
$C^3$ (where $C$ is the constant from Lemma \ref{lem:strict}), and we conclude that 
$\ell_{\gamma}(Y_{\gamma})\leq C^{3} \ell_{\gamma}(X_{\gamma})= C^{3} m_{\gamma}$.
\end{proof}

\begin{remark}
 We note here that $c=C^3$, where $C$ is the constant in Lemma \ref{lem:strict}. 
\end{remark}

\section{Model metrics and their properties}\label{sec:model}

\subsection{Building model metrics and finding such metrics close to a (thickened) designer metric}

{\bf Model metrics in $\M_g^{1/5n}$:} Set $n$ to be a positive integer, and let $P=\{\alpha_1,...,\alpha_{3g-3} \}$ be a pants decomposition of $\Sigma_{g}$. Let $\Vec{k}=(k_{1},...,k_{3g-3})$ where $0 \leq k_i \leq n$, for $i=1,...,3g-3$. We define
 $Z^{P}_{\vec{k}}$ to be the metric in Fenchel-Nielsen coordinates having zero twists and pants lengths
\begin{equation}
\ell_{\alpha_{i}}(Z^{P}_{\vec{k}})=\frac{k_{i}}{n} \text{ for } 0 \leq k_{i} \leq n.
\end{equation}
By convention $k_i=0$ corresponds to $\alpha_i$ representing a cusp. In such a situation the corresponding model metric is a noded hyperbolic surface on the boundary of moduli space.

Considering all such constructions with pants decompositions in different mapping class group orbits yields the set of model metrics, $\model^n$. 
That is

\begin{equation*}
\model^n := \left\{Z^{P}_{\vec{k}} | \text{ P a $\MCG_g$ orbit of a pants decomp., } \vec{k}=(k_1,...,k_{3g-3}),\, 0 \leq k_i \leq n \right\}.
\end{equation*}

Though we have a great number of model metrics, we can more or less count them: for each topological type of pants decomposition, we have $(n+1)^{3g-3}$ model metrics. 

\begin{figure}[h] 
 \centering
 \begin{overpic}[trim={1cm 10cm 1cm 10cm}, clip, width=0.6\linewidth]{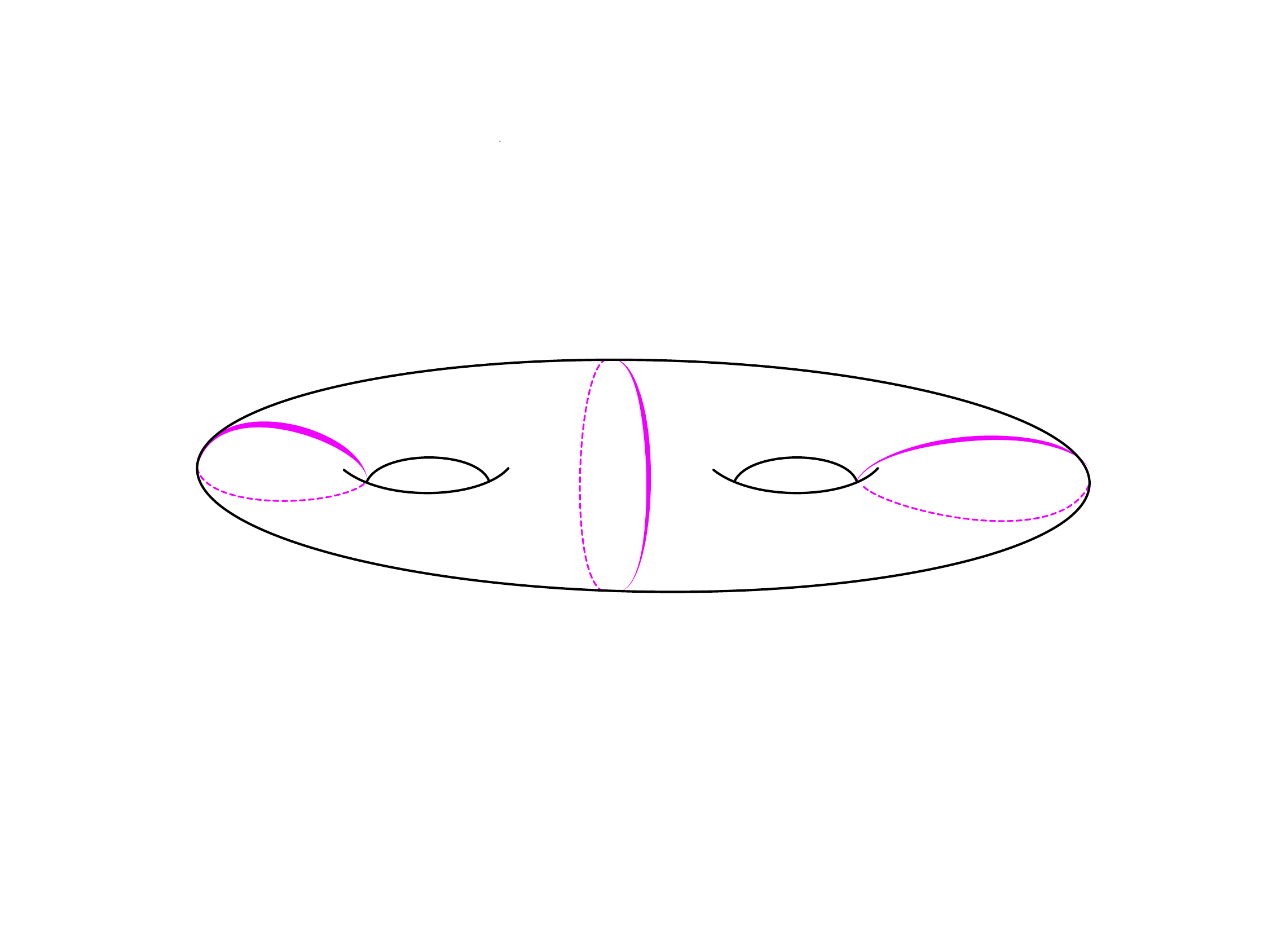}
 \put(20,12){$\alpha_1$}
 \put(77,11){$\alpha_3$}
 \put(46,15){$\alpha_2$}
 \end{overpic} 
 \caption{Pants decomposition of a genus $2$ surface} 
 \label{fig:pants decomp}
\end{figure}

{\noindent}{\bf Example:} Consider the pants decomposition 
$P=\{\alpha_{1}, \alpha_{2}, \alpha_3\}$ of the genus $2$ surface depicted in Figure 
\ref{fig:pants decomp}.
Set $n=3$ and consider the $\frac{1}{15}$-thick moduli space, $\M_2^{1/15}$. Then the set of model metrics for this pants decomposition have length parameters 
$$
\ell_{\alpha_{i}}(Z^{P}_{\vec{k}})=\frac{k_{i}}{n} 
$$
where $\vec{k}=(k_{1},k_{2},k_{3})$, for $k_{i} \in \{0,1,2,3\}$. 
There are $4^{3}$ model metrics corresponding to this pants decomposition. For example,

\begin{equation}
\ell_{\alpha_{i}}\left(Z^{P}_{(1,1,2)}\right)= 
\begin{cases}
 \frac{1}{3}, \text{ for } i=1,2\\
 \frac{2}{3}, \text{ for } i=3 \\
\end{cases}
\end{equation}

In the next lemma we show the main technical result we need, namely that any thickened metric for a curve with at most $n$ self-intersections is close to a model metric in $\model^n$.

\begin{lemma}[Model metric properties]\label{lem: model metric properties}
Let $\model$ be the set of all model metrics, and $Z^{P}_{\vec{k}}$ be a model metric associated to the pants decomposition
$P$.
\begin{enumerate}
\item If all the coordinates of $\vec{k}$ are non-zero, then the model metric $Z^{P}_{\vec{k}}$ is contained in $\M_g^{1/5n}$
\item If a coordinate of $\vec{k}$ is zero, then $Z^{P}_{\vec{k}}$ is a surface with nodes on the boundary of $\M_g$. Moreover, the number of such zeros correspond to the number of nodes on the hyperbolic surface. 

\item The number of mapping class group orbits of pants decompositions for a genus $g \geq 2$ surface, denoted $t_{g}$, only depends on the 
topology of the surface. Moreover, $t_g$ is the number of trivalent graphs with 
exactly $2g-2$ vertices.

\item \label{item: modelmetricnumber} the number of model metrics in 
$\model^{n}$ is $t_{g} (n+1)^{(3g-3)}$.

\item \label{item: qc constant} There exists a constant $Q_g$, which only depends on $g$, such that any surface in $\M_g^{1/5n}$ is $Q_g$-quasiconformally equivalent to a model metric in $\model^n$. 
Similarly, any metric in the $\frac{1}{5n}$-thick part of the boundary of $\M_g$ is $Q_g$-quasiconformally equivalent to a model metric in $\model^{n}$.

\item any metric in $\model^{n-1}$ is $Q_g$-quasiconformally equivalent to some model metric in $\model^{n}$.
\end{enumerate}
\end{lemma}

\begin{proof} To prove item (1) need to show that the systole of $Z^{P}_{\vec{k}}$ is $\geq \frac{1}{5n}$. Suppose not. 
Since the pants curves are bounded from above $B_{g}$, the systole cannot cross a pants curve and have length less than $\frac{1}{5n}$. This only leaves that the systole is a pants curve but by assumption all the pants curves have length a positive integer multiple of $\frac{1}{n}$. Thus it must be that the systole is bigger than $\frac{1}{5n}$. Item (2) follows from the definitions. 

To prove item (3), we first note that a pants decomposition can be identified 
with a $3$-regular graph with $2g-2$ vertices by putting a vertex for each pair of pants and connecting two vertices with an edge if they share a pants curve. 
Moreover, this correspondence uniquely identifies the mapping class orbit of the 
pants decomposition. Hence, the number of pants decomposition orbits is the number of trivalent graphs, $t_g$.
 
 Item (4) follows from the fact that there are $t_g$ pants decompositions up to an element of the mapping class group, and the number of 
 possible values for the length of any pants curve is $(n+1)^{3g-3}$.

To prove item (5), note that by Theorem \ref{thm:Bers} any surface $X \in \M_g^{1/5n}$ admits a pants decomposition
$P=\{\alpha_1,...,\alpha_{3g-3} \}$ with length upper bounded by $B_g$ and which contains all curves of length less than $2\arcsinh(1)$. We need to find an appropriate model metric with respect to $P$. This of course comes down 
to an appropriate choice of the vector $\vec{k}$.

Now if $\ell_{\alpha_{i}}(X) \in I$, where $I$ is one of the following intervals
\begin{equation*}
 \left[ \frac{1}{5n},\frac{1}{n} \right], \left(\frac{1}{n}, \frac{2}{n}\right],...,
\left(\frac{k-1}{n}, \frac{k}{n}\right],..., \left(\frac{n-1}{n}, 1 \right].
\end{equation*}
then choose $k_{i}$ so that $\frac{k_{i}}{n}$ is the biggest value in $I$.
Otherwise, $I=[1, B_g]$ and we set $k_{i}=n$ so that the model metric length coordinate is
$\ell_{\alpha_{i}}(Z^{P}_{\vec{k}})=1$.
 Setting all the twist parameters to be zero, this defines a model metric $Z^{P}_{\vec{k}}$. Now,
if $X_0$ has the same length parameters as $X$ but with zero twists, then the pants lengths from 
$Z^{P}_{\vec{k}}$ to $X_{0}$ are multiplied by at most 
$\text{max}\{B_g, 5\}$. Moreover, going from $X_0$ to $X$ involves 
twisting along each $\alpha_i$. Noting that the twists are locally supported about their respective curves, and that this mapping can be realized by a Nielsen twist about each curve with dilatation $K_i$,
for $i=1,...,3g-3$, we can conclude there is a quasiconformal mapping from $X_0$ to $X$ with dilatation bounded 
by the maximum of $\{K_1,...,K_{3g-3}\}$. The dilatations $K_i$ only depend on the upper bound of the length of the curve which is $B_g.$ See \cite{Matsuzaki2007} for a proof of this fact. 
Thus, by a theorem of Bishop (Proposition \ref{prop:bishop}) the map from $Z^{P}_{\vec{k}}$ to $X$ has dilatation at most 
$Q_g=f(\text{max}\{B_g, 5\}, B_g)$. The double function is because 
$\text{max}\{B_g, 5\}$ is  the upper bound on the length ratio and $B_g$ is the  upper bound on the pants lengths. Thus, the dilatation of the map from $Z^{P}_{\vec{k}}$ to $X$ depends only on the genus $g$. The argument for $X$ on the boundary of moduli space is similar and left to the reader. 

Finally, item (6) follows from item (\ref{item: qc constant}) and the fact that a model metric in $\model^{n-1}$ lies in the moduli space 
$\M_g^{1/5(n-1)} \subset \M_g^{1/5n}$.
\end{proof}

\begin{remark}
We remark that upper bounds as well as the asymptotic growth rate of $t_g$, as $g \rightarrow \infty$, are known. See Bollab\'as \cite{Bollobas1982}.
 
\end{remark}

\section{Proof of the inf spectrum growth}\label{sec:inf}

In this section, we prove Theorem \ref{theoremA}.

\begin{proof}[Proof of Theorem \ref{theoremA}]
For the lower bound, take any topologically embedded pair of pants $P \subset \Sigma$, and consider the set of closed curves on $P$. Note that the inf length of any curve in $P$ is realized on the ideal pair of pants and hence the set of inf lengths of these curves by Huber's theorem grow like $\frac{e^{L}}{2L}$. Since the (orientation preserving)  mapping class group of $P$ is of order $6$, we can conclude that the growth rate of these curves is at least 
$\frac{1}{6} (\frac{e^{L}}{2L})$. 
Now, if $P^{\prime}$ is a pair of pants not in the $\MCG_{g}$-orbit of $P$, then by 
 Lemma \ref{lem:mcgorbits}, item (2), the closed curves on $P^{\prime}$ are distinct from the closed curves on $P$. Letting $A_g$ be the number of $\MCG_{g}$-orbits of punctured tori (there is only one)
and pants in $\Sigma$, we may conclude 
 that 

\begin{equation*}
 \# \bigl\{[\gamma] \mid m_{[\gamma]} \leq L \bigr\}\geq A_g \frac{e^{L}}{2L} 
\end{equation*}

for $L \geq L_{0}$. This establishes the lower bound. 

For the proof of the upper bound, recall that $Y_{\gamma}$ denotes the thickened surface associated to the optimal metric, $X_{\gamma}$, with 
$\sys (Y_{\gamma}) \geq \frac{1}{5 i(\gamma, \gamma)}$. We use the fact that the number of self-intersections of a closed geodesic of length at most $L$ is universally bounded by $2e^{2L}$ (see \cite{basmajian2013universal}). Setting $n= 2e^{2L}$ we have 
\begin{equation} \label{eq: equality}
\{[\gamma]\mid \ell_{\gamma}(Y_{\gamma}) \leq L \} = \bigcup_{i(\gamma, \gamma)=1}^{n} \{[\gamma]\mid \ell_{\gamma}(Y_{\gamma})\leq L\} 
\end{equation}
\begin{equation}\label{eq: containment}
\subset \bigcup_{i(\gamma , \gamma)=1}^{n} \bigcup_{Z \in \model^{n}} \{[\gamma]\mid \frac{1}{Q_{g}} \ell_{\gamma}(Z) \leq L \} 
\end{equation}
 Line (\ref{eq: containment}) follows from Lemma \ref{lem: model metric properties} item (\ref{item: qc constant}), which guarantees that for any $Y_{\gamma}$ there is a model metric in $\model^{n}$ and a $Q_{g}$- quasiconformal mapping between them. 
Now, using the above containment we have

\begin{equation}\label{eq: equality2}
\#\{[\gamma]\mid \ell_{\gamma}(Y_{\gamma}) \leq L \}
= \sum_{i(\gamma, \gamma)=1}^{n} \#\{[\gamma]\mid \ell_{\gamma}(Y_{\gamma})\leq L\}
\end{equation}

\begin{align}
\label{eq: buser lemma}
\leq{}& \sum_{i(\gamma, \gamma)=1}^{n} \sum_{Z \in \model^{n}}
\#\{[\gamma]\mid \ell_{\gamma}(Z_{\gamma})\leq Q_g L\}
\leq \sum_{i(\gamma, \gamma)=1}^{n} \sum_{Z \in \model^{n}} a e^{Q_g L}
\\
\label{eq:numbermodelmetrics}
\leq{}& a e^{Q_g L} \sum_{i(\gamma, \gamma)=1}^{n} \sum_{Z \in \model^{n}} 1
= a ne^{Q_g L} \, \#\{Z \in \model^{n}\}
= a ne^{Q_g L} t_{g} (n+1)^{3g-3}
\\
\label{eq: final inequality}
={}& 2a t_{g} e^{(2+Q_g)L} 2^{3g-3}e^{(6g-6)L}
\left(1+\frac{1}{2e^{2L}}\right)^{3g-3}
\leq 2^{6g-5} a t_{g} e^{(6g-4+Q_g)L}
\end{align}

where the right-hand side of inequality (\ref{eq: buser lemma}) with constant 
$a=(g-1)e^{6}$ follows from Lemma 6.6.4 of Buser \cite{BuserBook}, in line (\ref{eq:numbermodelmetrics}) we used Lemma \ref{lem: model metric properties} item (\ref{item: modelmetricnumber}), and in line 
(\ref{eq: final inequality}) we used $n=2e^{2L}$, and in that same line the equality follows from using the fact that $1+\frac{1}{2e^{2L}} \leq 2$.

Finally, recalling Proposition \ref{prop: thickened metric} and plugging into line (\ref{eq: final inequality}) with $a=(g-1)e^{6}$ we have

\begin{equation}
\#\{[\gamma] \mid \ell_{\gamma}(X_{\gamma}) \leq L\} \leq \#\{[\gamma] \mid \ell_{\gamma}(Y_{\gamma}) \leq cL\} 
\end{equation}
\begin{equation*}
\leq (g-1)e^6 t_g 2^{6g-5} e^{(6g-4+Q_g)cL}
\end{equation*}
 
Finally, setting $B=(g-1)e^6 t_g 2^{6g-5}$ and $C=(6g-4+Q_g)c$, and noting that $B$
and $C$ only depend on topology yields the desired result. 
\end{proof}

\section{Counting embedded pairs of pants and figure 8 curves}\label{section:counting figure 8's}

In this section, we count embedded pairs of pants and figure 8 curves allowing us to complete the topological bounds on the constants $A,B$ and $C$ in Theorem 
\ref{theoremA}.

\subsection{Counting} 
Let $\Sigma=\Sigma_g$ be a closed surface of genus $g \geq 2$.
For the purpose of the next lemma, by an embedded compact pair of pants $P \subset S$ we mean a subsurface homeomorphic to a sphere with 3 open discs removed. 
A figure $8$ is a closed curve with 1 self-intersection. Associated to each figure $8$ on 
$\Sigma$ is a unique pair of pants with distinguished boundary curve. The distinguished boundary curve is the only boundary curve of the pants which is not enclosed by one of the loops of the figure $8$ (see Figure \ref{fig:figure8}).

\begin{figure}[h]
 \centering
 \begin{overpic}[trim={3cm 8cm 3cm 7cm}, clip, width=0.55\linewidth]{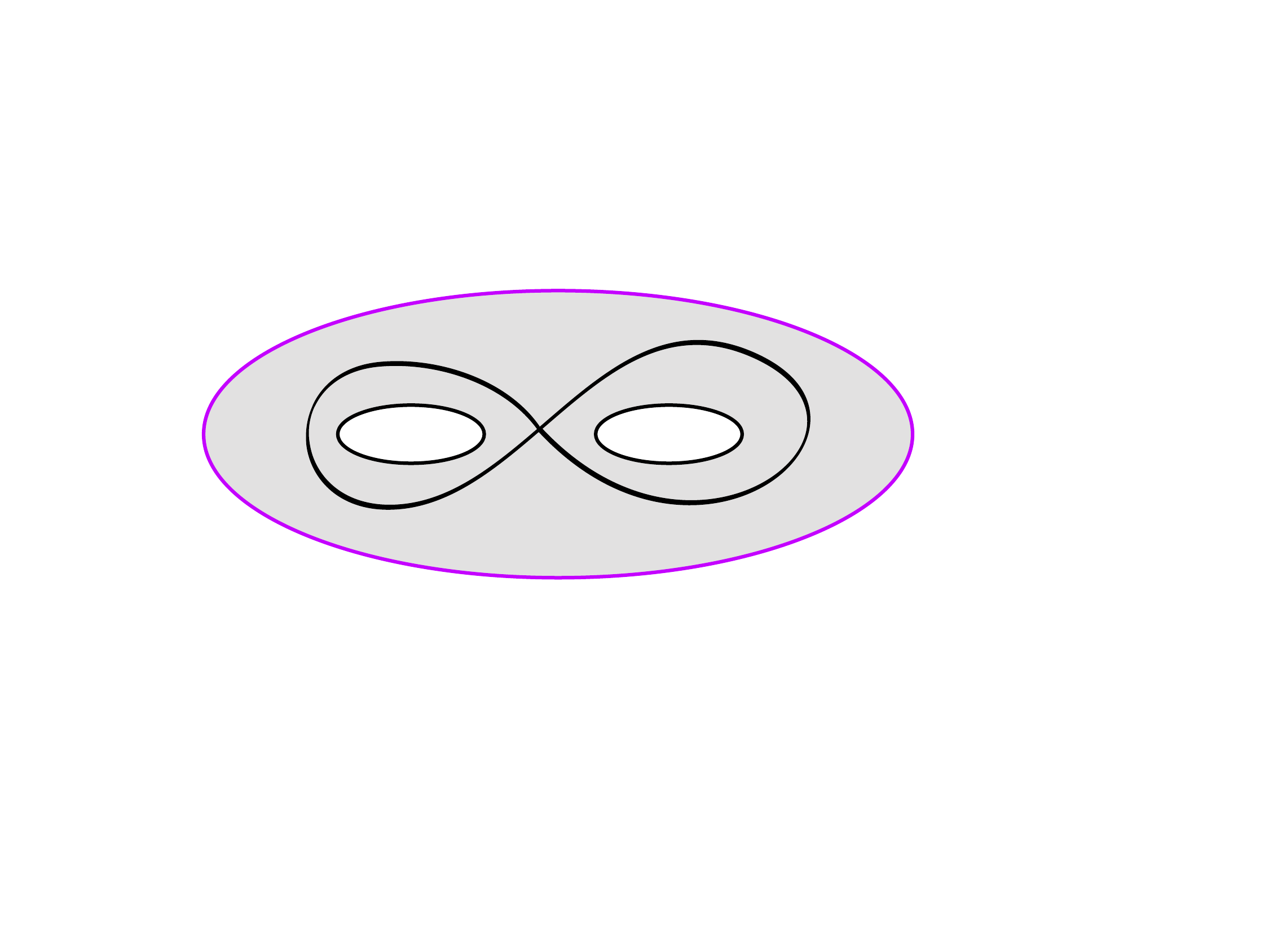}
 \put(72,13){$a$}
 \end{overpic}
 \caption{Figure 8 curve on a pair of pants with a distinguished boundary curve $a$ }
 \label{fig:figure8}
\end{figure}

\begin{lemma} \label{lem:mcgorbits}
 Suppose $\Sigma$ is a closed surface of genus $g \geq 2$. Then

\begin{enumerate}
 \item there is only one type of embedded punctured torus in $\Sigma.$
 Moreover, there is only one type of figure $8$ which is contained in the punctured torus. 
 \item Two essential closed curves in embedded pairs of pants, $\gamma \subset P$ and $\gamma^{\prime} \subset P^{\prime}$, are in the same orbit if and only if $P$ and $P^{\prime}$ are in the same orbit. 
 \end{enumerate}
\end{lemma}

\begin{proof} The fact that there is only one type of punctured torus follows from the fact that the complement of the punctured torus 
must have genus $g-1$ with one boundary component. After cutting along the nonseparating curve of the torus what remains is an embedded pair of pants which has only one figure $8$ up to an element of the mapping class group that keeps the nonseparating curve of the torus invariant. This verifies item (1).

 To prove item (2), note that the curve $\gamma$ deformation retracts onto a figure $8$ in $P$. Since figure $8$s track with the pants they are in we have a similar phenomenon for closed curves in pants. 
\end{proof}

\begin{table}[htbp]
	\centering
	\begin{tabular}{|c|c|c|}
		\hline \hline 
		
	{\bf Complementary components} & {\bf Embedded pants types} & {\bf Figure $8$ types} \\
		\hline
		\hline
 1 & 1 & 2 \\
 \hline
 2 & $(g-2)$ & $2(g-2)$ \\
 \hline
 3 & $\floor*{\frac{g^2+3}{12}}$ & $\floor*{\frac{(g-1)^2}{4}}$ \\
 \hline
 \hline
	\end{tabular} 
	\caption{Counting types} \label{table:2}
\end{table}

\begin{proposition}\label{enum: counting}
Consider the closed surface $\Sigma$ of genus $g \geq 2$. Then
\begin{enumerate} 
 \item the number of $\MCG_{g}$-orbits of embedded pairs of pants in $\Sigma$ is 
\begin{equation}
 (g-1)+ \floor*{\frac{g^2+3}{12}}
\end{equation}
 
 \item 
 the number of $\MCG_{g}$-orbits of figure $8$ curves in $\Sigma$ is 
 \begin{equation}
 2g-2+\floor*{\frac{(g-1)^2}{4}}
 \end{equation}
 
\end{enumerate}
\end{proposition}

\begin{proof}
Either count involves considering an embedded pair of pants $P$ in $\Sigma$. Such a pair of pants has either 1, 2, or 3 complementary components $\Sigma\backslash P$. We count the number of such complementary components in each case. 

\noindent{\bf 1 complementary component:} 
In this case, an embedded pair of pants has complement a genus $g-2$ surface with 3 boundary components. See Figure \ref{fig:1 component}. Hence, there is 1 embedded pair of pants. Next to count figure 8's note that either the figure 8 is in a punctured torus or an embedded pair of pants. In the case of a punctured torus by Lemma \ref{lem:mcgorbits} item (1) there is a single figure 8. In the case of an embedded pair
of pants $P$, since any of the boundary curves of $P$ are equivalent, there is only 1 figure $8$ curve in an embedded pair of pants with one complementary component. We may conclude there are exactly 2 figure 8's in this case.

\begin{figure}[h]
 \centering
 \begin{overpic}[trim={2cm 7cm 8cm 6cm}, clip, width=0.65\linewidth]{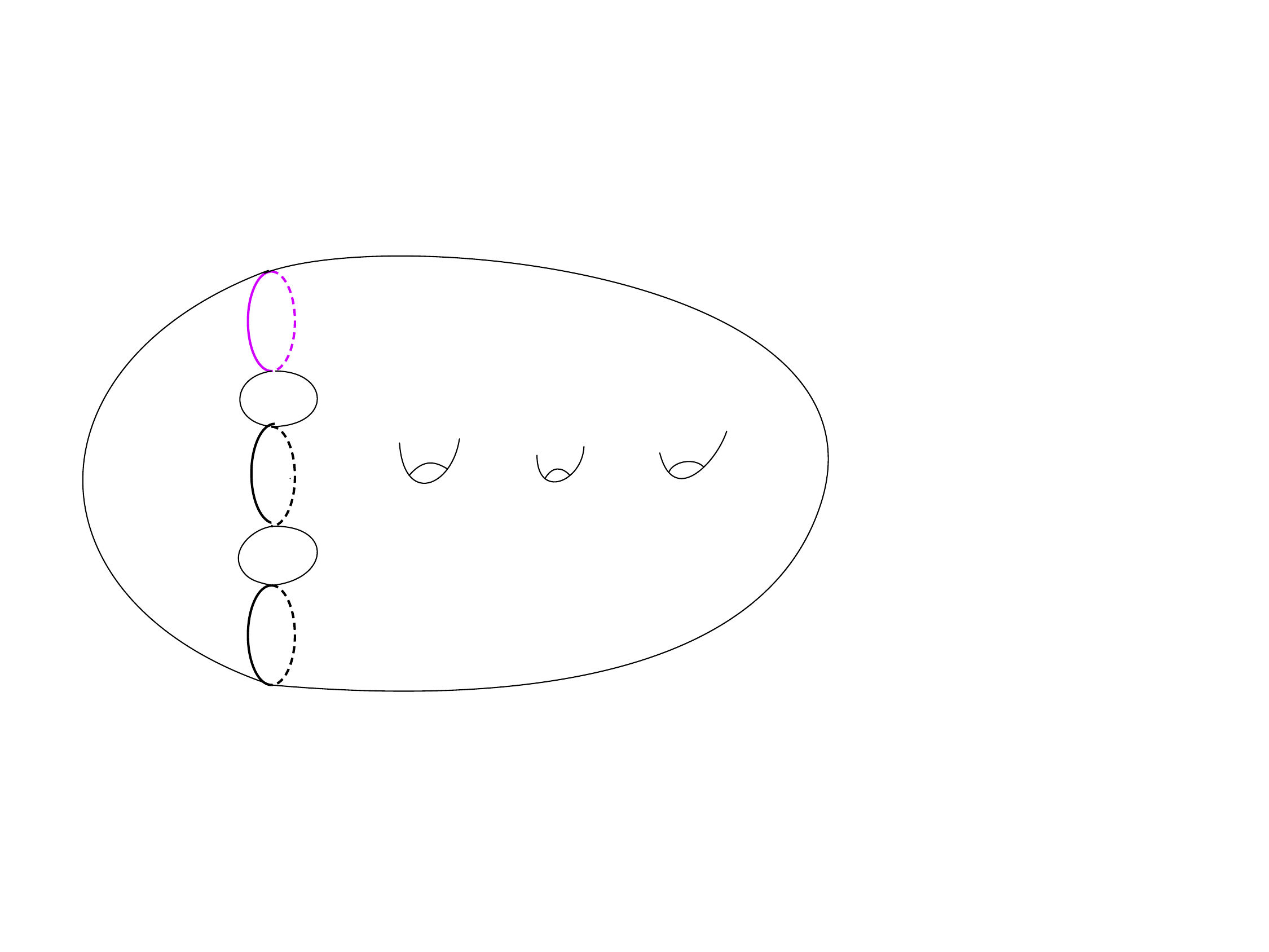}
 \put(25,41){$a$}
 \end{overpic}
 \caption{Embedded pair of pants with 1 complementary component and distinguished boundary $a$. }
 \label{fig:1 component}
\end{figure}

\noindent{\bf 2 complementary components:} In this case, one of the components has a single boundary curve, and the other complementary component has 2 boundary curves. See Figure \ref{fig:2 component both}. Denote the genus of the complementary component with one boundary curve by $a$, and denote the genus of the other complementary component as $g_1$. Now $a+g_1=g-1$, where $a \geq 1$ and $g_1 \geq 1.$ There are clearly $(g-2)$ possible configurations, and hence the number of embedded pairs of pants is 
$g-2$. 

To count the figure $8's$ there are two possibilities for the distinguished boundary component. Either it bounds the component with genus $a$ which has one boundary component or it bounds the component with genus $g_1$ with two boundary components. In the first case, there are 
$(g-2)$ such configurations. In the second case, there are also $(g-2)$ such possibilities. Hence, the total number of figure $8's$ is $2(g-2)$.

\begin{figure}[h]
 \centering
 \begin{subfigure}[b]{0.45\linewidth}
 \centering
 \begin{overpic}[trim={2cm 7cm 8cm 6cm}, clip, width=\linewidth]{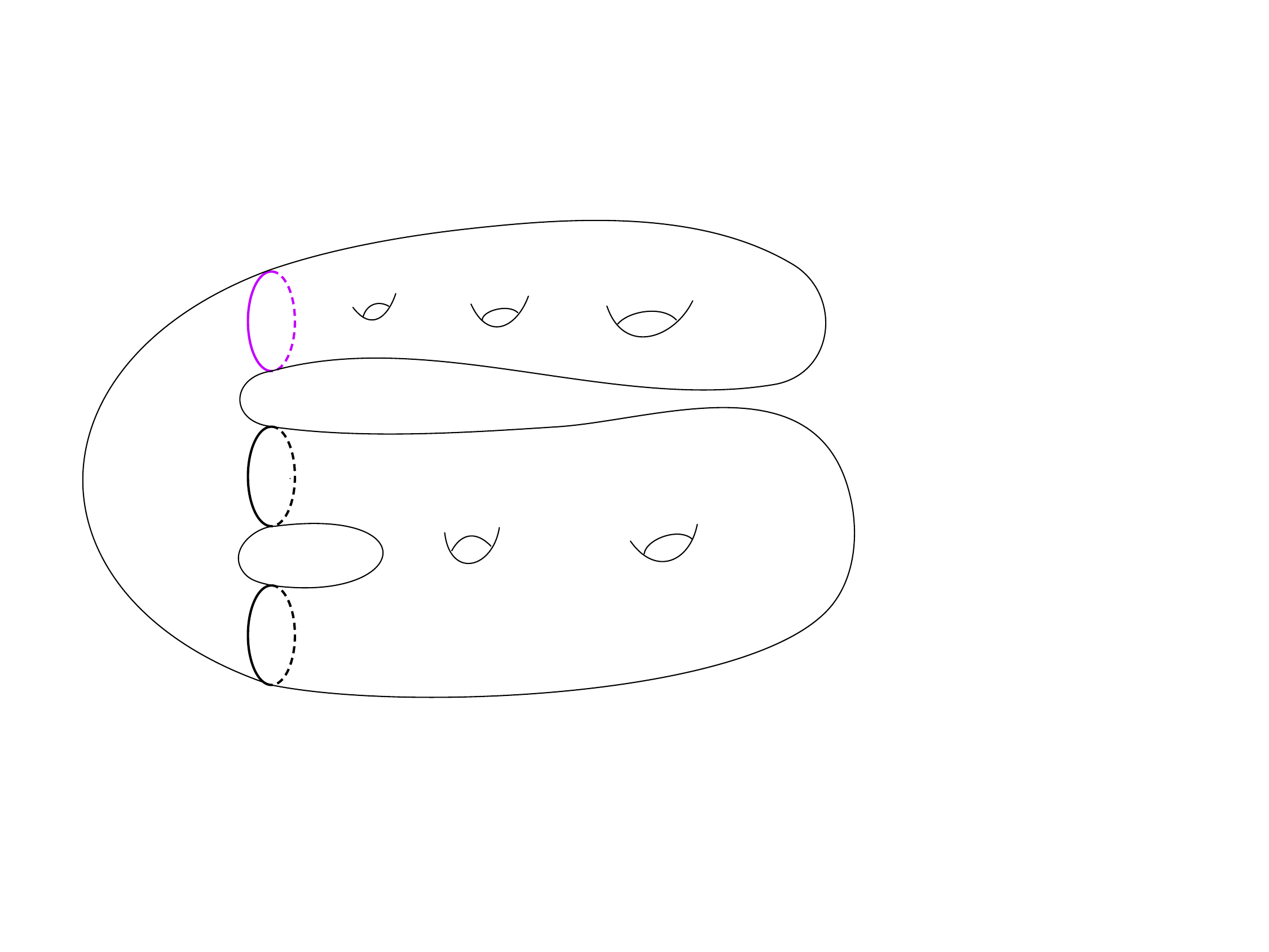}
 \put(25,41){$a$}
 \end{overpic}
 \caption{}
 \label{fig:2 component}
 \end{subfigure}
 \hfill
 \begin{subfigure}[b]{0.45\linewidth}
 \centering
 \begin{overpic}[trim={2cm 7cm 8cm 6cm}, clip, width=\linewidth]{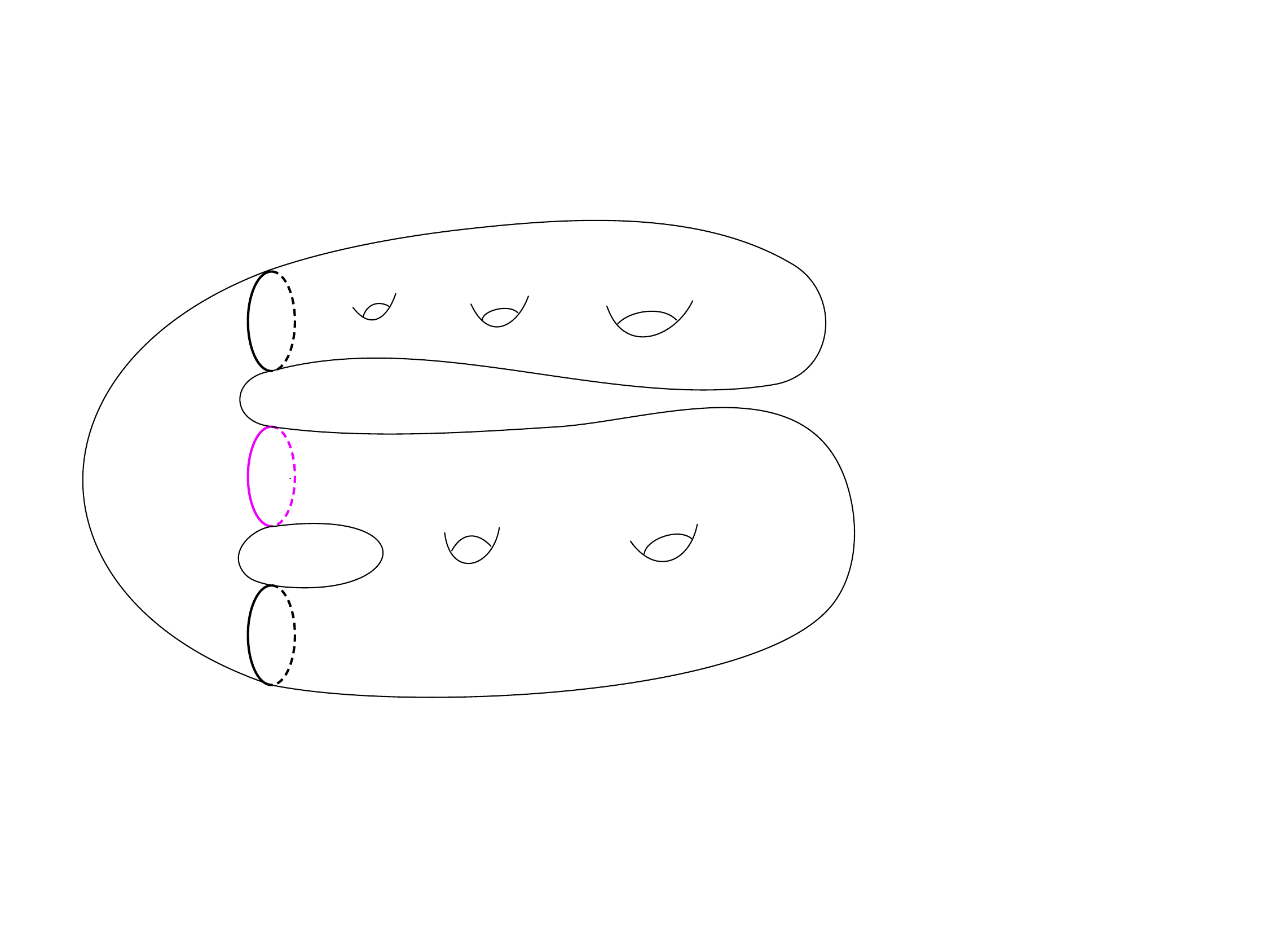}
 \put(25,41){$a$}
 \end{overpic}
 \caption{}
 \label{fig:2component2}
 \end{subfigure}
 \caption{Embedded pair of pants with 2 complementary components and distinguished boundary $a$.}
 \label{fig:2 component both}
\end{figure}

\noindent{\bf 3 complementary components:} Here we assume $g \geq 3$ since 3 complementary components is impossible if genus $g=2$. Each of these complementary components has positive genus and one boundary component.
See Figure \ref{fig:3 component}. Denoting these genera by $a,g_1$, and $g_2$ it is not difficult to see that the total number of embedded pants is equal to

$$
\# \{(a,g_1,g_2) | a+g_1+g_2=g, 1 \leq a \leq g_1 \leq g_2 \}
$$
but this is precisely the number of positive partitions of $g$ with exactly 3 parts which is
$$
\floor*{\frac{g^2+3}{12}}.
$$

To count the number of figure $8$ curves in this case, let $a$ be the genus with the distinguished boundary component corresponding to the figure $8$ on the pair of pants, and let 
$g_1$ and $g_2$ be the genera of the other two components. 
We know $a+g_1+g_2=g$, where $1 \leq a \leq g-2$.
Now for each $a$, note that $g_1$ and $g_2$ are interchangeable since they yield the same surface $\Sigma$ . Hence for each $a$ there are 
$$
\floor*{\frac{g-a}{2}}.
$$
Since $a$ runs between $1$ and $g-2$, we have that the total number of figure $8's$ in this case is 
\begin{equation}
 \sum_{a=1}^{g-2} \floor*{\frac{g-a}{2}}=\floor*{\frac{(g-1)^2}{4}}
\end{equation}

At this point, the proof, in any of the cases reduces to an elementary combinatorial exercise. 
We leave the count to the reader with Table \ref{table:2} providing the count for the various possible complementary components.

\end{proof}

\begin{figure}[h]
 \centering
 \begin{overpic}[trim={2cm 7cm 8cm 6cm}, clip, width=0.65\linewidth]{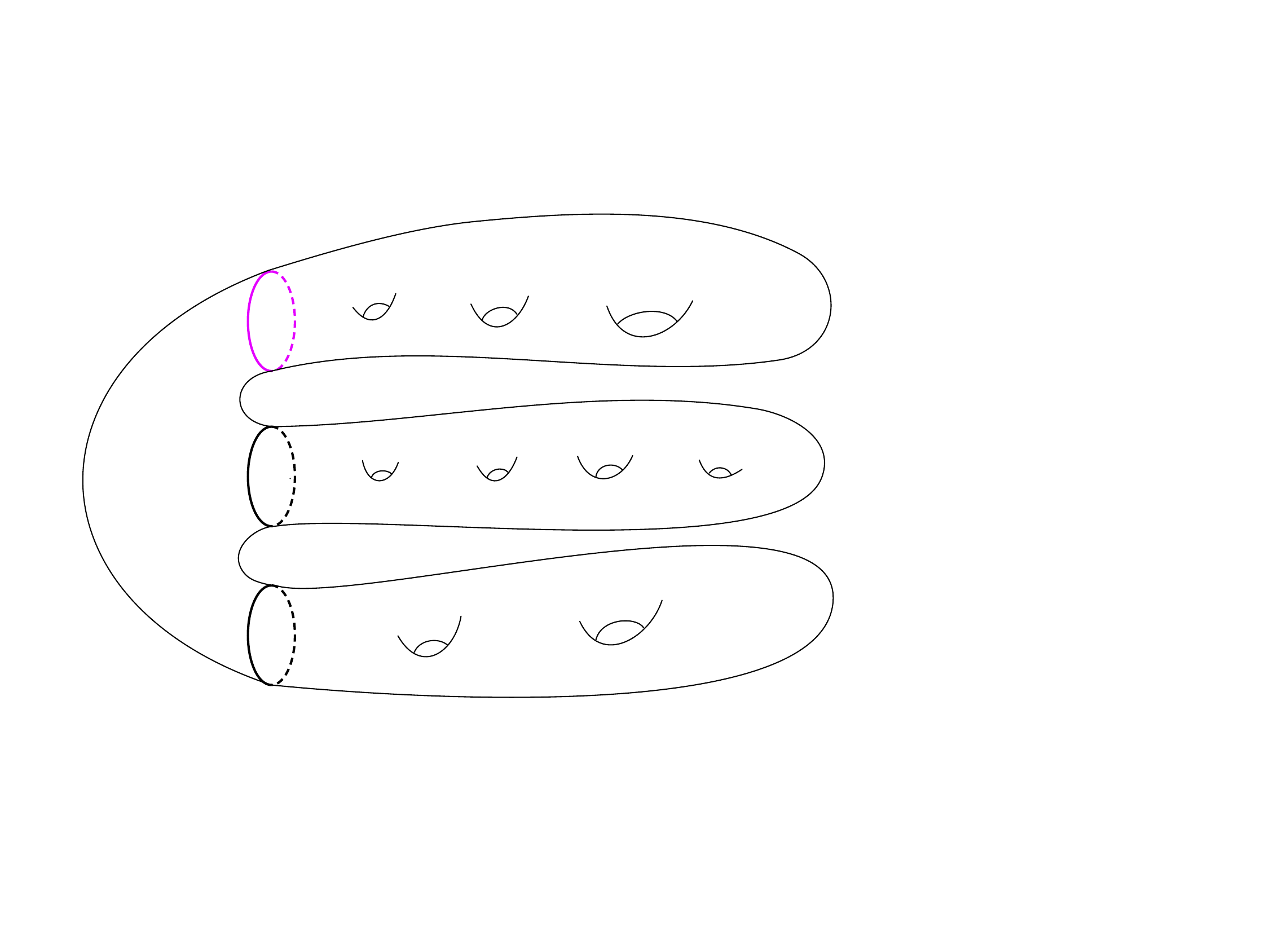}
 \put(25,41){$a$}
 \end{overpic}
 \caption{Embedded pair of pants with 3 complementary components and distinguished boundary $a$. }
 \label{fig:3 component}
\end{figure}

\section{Properties of the Inf Spectrum}\label{sec:propertiesofinfspec}

\begin{proposition}\label{prop: infspecproperties}
 Let $\Sigma$ be a closed surface of genus $g \geq 2$.
\begin{enumerate}

\item The first two numbers in the inf spectrum are $0$ and 
$2\arccosh \, 3$. 

\item The multiplicity of $0$ is the number of mapping class group orbits of simple curves which equals $\floor{\frac{g}{2}} +1$.

\item The multiplicity of 
$2\arccosh \,3$ is the number of mapping class group orbits of figure 8 curves which is
the increasing function 
\begin{equation}
 2g-2+\floor*{\frac{(g-1)^2}{4}}.
 \end{equation}
In particular, the inf spectrum determines the topology of $\Sigma$. 

\item If $\Sigma^{\prime}$ is a subsurface of $\Sigma$ then 
$\infspec(\M')$ injects into $\infspec(\M)$. 

\item The primitive length spectrum of an ideal pair of pants injects into $\infspec(\M)$. 
In particular:
$$
\{2\arccosh (2k+1)\}_{k=1}^{\infty} \hookrightarrow \infspec(\M).
$$
\end{enumerate}
\end{proposition}

We remark here that the injection in items (4) and (5) can be far from surjective. That is,
the multiplicity of an inf length can change. As can be seen for simple curves and figure 
$8$s.

\begin{proof} 
The inf length of an essential simple closed curve is $0$, and the figure $8$ being the shortest non-simple curve (\cite{BuserBook}) is guaranteed to be the next shortest inf length. The proof of item (1) is finished by noting that a figure $8$ has inf length equal to $2\arccosh \, 3$. The proof of item (2) is clear. The proof of item (3) follows from Proposition
\ref{enum: counting}. Item (4) follows from the fact that any curve in 
$\Sigma'$ is also a curve in $\Sigma$, and the fact that the inf length of a curve on a subsurface is the same as its inf length on any other surface that contains $\Sigma'$. Finally, item (5) follows from the fact that the inf length of a curve on an embedded pair of pants is exactly the length of that curve on an ideal pair of pants. 
\end{proof}

\begin{remark}
In item (4) of Proposition \ref{prop: infspecproperties} the multiplicity of the inf length of a curve can change depending on the surface it is contained in. For example, the multiplicity of figure $8$s changes in going to a bigger surface.
\end{remark}

{\bf Comparison with the length spectrum :} It is interesting to note that there is no value in the length spectrum of $X$ for which the curve type is determined. 
Even for the smallest length (that is, the systole) the only thing one can discern is that it is a simple curve on $X$. On the other hand, for the Inf spectrum some curve types are determined. We point out that the lengths of corkscrew geodesics 
$\{M_{k} \}$ for $k$ large are known to be the shortest geodesics (among all surfaces and all curves of self-intersection number at least $k$) \cite{MR4672786, MR4460231}.
On the other hand, the values $\{M_{k}\}$ may have multiplicity. The multiplicity can arise from other corkscrew geodesics in a pair of pants that are not in the same 
mapping class group orbit or they may come in a much more exotic fashion. Namely, we cannot rule out the possibility that there is a curve with less than $k$ self-intersections whose inf length equals $M_k$.

\section{Growth bounds on the Inf Spectrum using intersection growth}\label{sec:growth using intersection}

The Aougab-Souto Theorem \cite{souto2018counting} on the number of curve types with self-intersection less than $K$ yields the following inequality (using, in their notation, $\delta =\frac{1}{2}$):
\begin{equation} \label{eq: intersection growth}
e^{(\pi \sqrt{|\chi (\Sigma)|}-\frac{1}{2})\sqrt{K}}
\leq \#\bigl\{[\gamma]: i(\gamma,\gamma) \leq K \bigr\}
\leq e^{(4 \sqrt{2|\chi (\Sigma)|}+\frac{1}{2})\sqrt{K}}
\end{equation}
for $K$ bigger than some constant. Furthermore, given a curve $\gamma$, its optimal metric $X_{\gamma}$ satisfies
\begin{equation}\label{eq: optimal metric bound}
 \frac{1}{2} \log \frac{i(\gamma,\gamma)}{2} \leq \ell_{\gamma}(X_{\gamma}) 
 \leq c_1 \sqrt{i(\gamma,\gamma)}
\end{equation}
where $c_1=4 \sqrt{2|\chi (\Sigma)|}$. The lower bound on 
$\ell_{\gamma} (X_{\gamma})$ follows from \cite{basmajian2013universal} and the upper bound follows from \cite{souto2018counting}. Moreover, there exist infinite families of curve types whose optimal metrics exhibit logarithmic growth, and on the other extreme there exist infinite families of curve types whose optimal metrics exhibit square root growth. See \cite{BM} for these examples. 
Now, given the inequalities (\ref{eq: intersection growth}) and 
(\ref{eq: optimal metric bound}), it is natural to ask if they lead to good bounds on the growth rate of the Inf spectrum. For the lower bound an elementary computation yields
\begin{equation*}
\#\bigl\{[\gamma]: m_{[\gamma ]} \leq L \bigr\}
\geq \#\bigl\{[\gamma]: c_1 \sqrt{i(\gamma,\gamma)} \leq L \bigr\}
= \#\bigl\{[\gamma]: i(\gamma,\gamma) \leq \frac{L^2}{c_1^2} \bigr\}
\geq e^{bL}
\end{equation*}
where $b=\frac{\pi\sqrt{|\chi(\Sigma)|}-\frac{1}{2}}{4\sqrt{2}
\sqrt{|\chi(\Sigma)|}} <1.$ For the upper bound one gets double exponential growth as indicated in
 \cite{BM}. 
Thus, a straightforward application of these bounds produces a weaker 
exponential lower bound and considerably weaker upper bound than the ones in Theorem \ref{theoremA}. 

\vskip5pt

\centerline{\bf Statement about the use of AI}

Beyond minimal proofreading, no AI was used to prepare this manuscript. All ideas, math and writing were done by us.

\bibliographystyle{plain} 
\bibliography{references} 

{\it Addresses:} 
Graduate Center and Hunter College, CUNY, New York, USA\\
Fordham University, New York, USA\\
University of Fribourg, Switzerland

{\it Emails:} abasmajian@gc.cuny.edu\\
smondal6@fordham.edu\\
hugo.parlier@unifr.ch

\end{document}